\documentclass[]{amsart}
\pdfoutput=1

\usepackage[leftmargin=1em,rightmargin=1ex,vskip=1ex]{quoting}
\usepackage{lscape}
\usepackage{subcaption}
\usepackage{scalerel}
\usepackage{amssymb}
\usepackage{amsthm}
\usepackage{mathrsfs}
\usepackage{hyphenat}
\usepackage{dutchcal}
\usepackage{quiver} 
\usepackage{mlmodern}

\newcommand{\biobj}[2]{{\protect\substack{ #1 \\ #2 }}}

\usepackage[utf8]{inputenc}
\usepackage{cmll}

\makeatletter
\newcommand{\oset}[3][0ex]{%
  \mathrel{\mathop{#3}\limits^{
    \vbox to#1{\kern-2\ex@
    \hbox{$\scriptstyle#2$}\vss}}}}
\makeatother

\newcommand\scalemath[3]{\scalebox{#1}[#2]{\mbox{\ensuremath{\displaystyle #3}}}}

\newcommand{\leftarrowtip}{\ensuremath{\tikz\draw[line width=0.5pt,->] (10pt,0) -- (0,0);}}
\newcommand{\leftarrowtailnotip}{\ensuremath{\tikz\draw[line width=0.5pt,-<] (0,0) -- (10pt,0);}}

\newcommand{\unicodeStar}{\ensuremath{\star}}
\DeclareUnicodeCharacter{2605}{\unicodeStar}

\DeclareUnicodeCharacter{21D2}{\ensuremath{\Rightarrow}}
\DeclareUnicodeCharacter{2218}{\ensuremath{\circ}}
\DeclareUnicodeCharacter{2022}{\ensuremath{\bullet}}
\DeclareUnicodeCharacter{2219}{\ensuremath{\bullet}}
\DeclareUnicodeCharacter{2026}{\ensuremath{\dots}}
\DeclareUnicodeCharacter{2208}{\ensuremath{\in}}
\DeclareUnicodeCharacter{2192}{\ensuremath{\to}}
\DeclareUnicodeCharacter{2190}{\ensuremath{\leftarrowtip}}
\DeclareUnicodeCharacter{2919}{\ensuremath{\leftarrowtailnotip}}
\DeclareUnicodeCharacter{2270}{\ensuremath{\nleq}}
\DeclareUnicodeCharacter{2B47}{\ensuremath{\oset{\backsim}{\to}}}
\DeclareUnicodeCharacter{2291}{\ensuremath{\sqsubseteq}}
\DeclareUnicodeCharacter{22A0}{\ensuremath{\boxtimes}}
\DeclareUnicodeCharacter{22B0}{\ensuremath{\boxtimes}}
\DeclareUnicodeCharacter{1D4B9}{\ensuremath{\mathcal{d}}}
\DeclareUnicodeCharacter{1D4C8}{\ensuremath{\mathcal{s}}}

\DeclareUnicodeCharacter{00D7}{\ensuremath{\times}}
\DeclareUnicodeCharacter{00B7}{\ensuremath{\cdot}}
\DeclareUnicodeCharacter{222B}{\ensuremath{\int}}
\DeclareUnicodeCharacter{22A4}{\ensuremath{\top}}
\DeclareUnicodeCharacter{22A5}{\ensuremath{\bot}}
\DeclareUnicodeCharacter{2264}{\ensuremath{\leq}}
\DeclareUnicodeCharacter{2225}{\ensuremath{\parallel}}

\newcommand{\unicodecolon}{\ensuremath{\colon}}
\DeclareUnicodeCharacter{FE55}{\unicodecolon}
\newcommand{\unicodeleftpar}{\ensuremath{\left(}}
\DeclareUnicodeCharacter{27EE}{\unicodeleftpar}
\newcommand{\unicoderightpar}{\ensuremath{\right)}}
\DeclareUnicodeCharacter{27EF}{\unicoderightpar}
\DeclareUnicodeCharacter{2260}{\neq}
\DeclareUnicodeCharacter{22A9}{\Vdash}
\DeclareUnicodeCharacter{2237}{\proportion}
\DeclareUnicodeCharacter{2124}{\mathbb{Z}}
\DeclareUnicodeCharacter{27E8}{\langle}
\DeclareUnicodeCharacter{27E9}{\rangle}
\DeclareUnicodeCharacter{21A6}{\mapsto}
\DeclareUnicodeCharacter{22A2}{\vdash}
\DeclareUnicodeCharacter{2090}{\ensuremath{{}_a}}
\DeclareUnicodeCharacter{A71B}{{}^\uparrow}
\DeclareUnicodeCharacter{A71C}{{}^\downarrow}
\DeclareUnicodeCharacter{27E6}{\llbracket}
\DeclareUnicodeCharacter{27E7}{\rrbracket}

\newcommand{\unicoderightcircle}{\ensuremath{\RIGHTcircle}}
\DeclareUnicodeCharacter{25D1}{\unicoderightcircle}
\newcommand{\unicodeleftcircle}{\ensuremath{\LEFTcircle}}
\DeclareUnicodeCharacter{25D0}{\unicodeleftcircle}
\DeclareUnicodeCharacter{229B}{\circledast}

\newcommand{\unicodebbA}{\ensuremath{\mathbb{A}}}
\DeclareUnicodeCharacter{1D538}{\unicodebbA}
\newcommand{\unicodebbB}{\ensuremath{\mathbb{B}}}
\DeclareUnicodeCharacter{1D539}{\unicodebbB}
\newcommand{\unicodebbC}{\ensuremath{\mathbb{C}}}
\DeclareUnicodeCharacter{2102}{\unicodebbC}
\DeclareUnicodeCharacter{1D53B}{\ensuremath{\mathbb{D}}}
\DeclareUnicodeCharacter{2113}{\ensuremath{\ell}}
\DeclareUnicodeCharacter{2115}{\ensuremath{\mathbb{N}}}
\DeclareUnicodeCharacter{211D}{\ensuremath{\mathbb{R}}}
\DeclareUnicodeCharacter{1D543}{\ensuremath{\mathbb{L}}}
\newcommand\UnicodeBlackboardP{\ensuremath{\mathbf{P}}} \DeclareUnicodeCharacter{2119}{\UnicodeBlackboardP}
\DeclareUnicodeCharacter{211A}{\ensuremath{\mathbb{Q}}}
\DeclareUnicodeCharacter{1D544}{\ensuremath{\mathbb{M}}}
\DeclareUnicodeCharacter{1D54C}{\ensuremath{\mathbb{U}}}
\DeclareUnicodeCharacter{1D54D}{\ensuremath{\mathbb{V}}}
\DeclareUnicodeCharacter{1D54E}{\ensuremath{\mathbb{W}}}
\DeclareUnicodeCharacter{1D546}{\ensuremath{\mathbb{O}}}
\DeclareUnicodeCharacter{1D540}{\ensuremath{\mathbb{I}}}
\DeclareUnicodeCharacter{1D54A}{\ensuremath{\mathbb{S}}}
\DeclareUnicodeCharacter{1D53C}{\ensuremath{\mathbb{E}}}

\newcommand{\unicodecalS}{\ensuremath{\mathcal{S}}}
\newcommand{\unicodecalT}{\ensuremath{\mathcal{T}}}
\newcommand{\unicodecalC}{\ensuremath{\mathcal{C}}}
\newcommand{\unicodecalD}{\ensuremath{\mathcal{D}}}
\newcommand{\unicodecalX}{\ensuremath{\mathcal{X}}}
\newcommand{\unicodecalN}{\ensuremath{\mathcal{N}}}
\newcommand{\unicodecalE}{\ensuremath{\mathcal{E}}}
\DeclareUnicodeCharacter{1D4D4}{\unicodecalE}
\DeclareUnicodeCharacter{1D4D2}{\unicodecalC}
\DeclareUnicodeCharacter{1D4D3}{\unicodecalD}
\DeclareUnicodeCharacter{1D4DE}{\mathcal{O}}
\DeclareUnicodeCharacter{1D4E2}{\unicodecalS}
\DeclareUnicodeCharacter{1D4E3}{\unicodecalT}
\DeclareUnicodeCharacter{1D4E7}{\unicodecalX}
\DeclareUnicodeCharacter{1D4D0}{\ensuremath{\mathcal{A}}}
\DeclareUnicodeCharacter{1D4D1}{\ensuremath{\mathcal{B}}}
\DeclareUnicodeCharacter{1D4D6}{\ensuremath{\mathcal{G}}}
\DeclareUnicodeCharacter{1D4D7}{\ensuremath{\mathcal{H}}}
\DeclareUnicodeCharacter{1D4DB}{\ensuremath{\mathcal{L}}}
\DeclareUnicodeCharacter{1D4DC}{\ensuremath{\mathcal{M}}}
\DeclareUnicodeCharacter{1D4DD}{\unicodecalN}
\DeclareUnicodeCharacter{1D4B1}{\ensuremath{\mathcal{V}}}
\DeclareUnicodeCharacter{1D4E5}{\ensuremath{\mathcal{V}}}
\DeclareUnicodeCharacter{1D4E6}{\ensuremath{\mathcal{W}}}
\DeclareUnicodeCharacter{1D4E4}{\ensuremath{\mathcal{U}}}

\DeclareUnicodeCharacter{03B1}{\alpha}
\DeclareUnicodeCharacter{03B2}{\beta}
\DeclareUnicodeCharacter{03BC}{\mu}
\DeclareUnicodeCharacter{03B4}{\delta}
\DeclareUnicodeCharacter{03B5}{\varepsilon}
\DeclareUnicodeCharacter{03B7}{\eta}
\DeclareUnicodeCharacter{03BB}{\lambda}
\DeclareUnicodeCharacter{03C1}{\rho}
\DeclareUnicodeCharacter{03C8}{\psi}
\DeclareUnicodeCharacter{03C4}{\tau}
\DeclareUnicodeCharacter{03A8}{\Psi}
\DeclareUnicodeCharacter{03C3}{\sigma}
\DeclareUnicodeCharacter{03C6}{\varphi}
\DeclareUnicodeCharacter{03A6}{\Phi}
\DeclareUnicodeCharacter{03A3}{\Sigma}
\DeclareUnicodeCharacter{03A0}{\Pi}
\DeclareUnicodeCharacter{03D5}{\phi}
\DeclareUnicodeCharacter{03B8}{\theta}
\DeclareUnicodeCharacter{03C0}{\ensuremath{\pi}}
\DeclareUnicodeCharacter{0393}{\Gamma}
\DeclareUnicodeCharacter{0394}{\Delta}
\DeclareUnicodeCharacter{03BA}{\kappa}
\DeclareUnicodeCharacter{03BD}{\nu}
\DeclareUnicodeCharacter{25A0}{\blacksquare}
\DeclareUnicodeCharacter{25AA}{\blacksquare}

\newcommand{\hirayo}{\scaleobj{0.9}{\text{\usefont{U}{min}{m}{n}\symbol{'210}}}}
\DeclareUnicodeCharacter{3088}{\hirayo}
\DeclareFontFamily{U}{min}{}
\DeclareFontShape{U}{min}{m}{n}{<-> udmj30}{}

\newcommand\UnicodeWhiteRightPointingSmallTriangle{\triangleright}
\DeclareUnicodeCharacter{25B9}{\mathbin{\UnicodeWhiteRightPointingSmallTriangle}}
\newcommand\UnicodeWhiteDownPointingSmallTriangle{\triangledown}
\DeclareUnicodeCharacter{25BF}{\mathbin{\UnicodeWhiteDownPointingSmallTriangle}}
\newcommand\UnicodeWhiteUpPointingSmallTriangle{\scalemath{1}{-1}{{}^{\triangledown}}}
\DeclareUnicodeCharacter{25B5}{\mathbin{\UnicodeWhiteUpPointingSmallTriangle}}

\DeclareUnicodeCharacter{2080}{\ensuremath{{}_0}}
\DeclareUnicodeCharacter{2081}{\ensuremath{{}_1}}
\DeclareUnicodeCharacter{2082}{\ensuremath{{}_2}}
\DeclareUnicodeCharacter{2083}{\ensuremath{{}_3}}

\DeclareUnicodeCharacter{1D62}{\ensuremath{{}_i}}
\DeclareUnicodeCharacter{2C7C}{\ensuremath{{}_j}}
\DeclareUnicodeCharacter{02B3}{\ensuremath{{}^r}}
\DeclareUnicodeCharacter{02E1}{\ensuremath{{}^\ell}}
\DeclareUnicodeCharacter{1D48}{\ensuremath{{}^d}}
\DeclareUnicodeCharacter{1D50}{\ensuremath{{}^m}}
\DeclareUnicodeCharacter{1D58}{\ensuremath{{}^u}}
\DeclareUnicodeCharacter{209A}{\ensuremath{{}_p}}
\DeclareUnicodeCharacter{2096}{\ensuremath{{}_k}}
\DeclareUnicodeCharacter{209C}{\ensuremath{{}_t}}

\DeclareUnicodeCharacter{2245}{\ensuremath{\cong}}
\DeclareUnicodeCharacter{2286}{\subseteq}

\DeclareUnicodeCharacter{22C5}{\cdot}
\DeclareUnicodeCharacter{25C3}{\ensuremath{\triangleleft}}
\DeclareUnicodeCharacter{25B9}{\ensuremath{\triangleright}}

\newcommand\smallmath[2]{#1{\raisebox{\dimexpr \fontdimen 22 \textfont 2
      - \fontdimen 22 \scriptscriptfont 2 \relax}{$\scriptscriptstyle #2$}}}
\newcommand\smalloplus{\smallmath\mathbin\oplus}

\DeclareUnicodeCharacter{2295}{\smalloplus}
\DeclareUnicodeCharacter{2297}{\otimes}
\DeclareUnicodeCharacter{214B}{\parr}
\DeclareUnicodeCharacter{2298}{\oslash}
\DeclareUnicodeCharacter{25C0}{\mathbin{\blacktriangleleft}}
\DeclareUnicodeCharacter{25C1}{\mathbin{\vartriangleleft}}
\DeclareUnicodeCharacter{22B3}{\mathbin{\triangleright}}
\DeclareUnicodeCharacter{22B2}{\mathbin{\triangleleft}}
\DeclareUnicodeCharacter{FF5C}{\mid}
\DeclareUnicodeCharacter{227A}{\mathbin{\prec}}
\DeclareUnicodeCharacter{227B}{\mathbin{\succ}}
\DeclareUnicodeCharacter{22A3}{\mathbin{\dashv}}
\DeclareUnicodeCharacter{219D}{\ensuremath{\leadsto}}
\DeclareUnicodeCharacter{1361}{\colon}

\DeclareUnicodeCharacter{29D1}{\mathrel{\multimapdotbothB}}
\DeclareUnicodeCharacter{29D2}{\mathrel{\multimapdotbothA}}
\DeclareUnicodeCharacter{22C4}{\mathbin{\diamond}}
\DeclareUnicodeCharacter{226B}{\mathrel{\gg}}
\DeclareUnicodeCharacter{25A1}{\Box}
\DeclareUnicodeCharacter{266F}{\sharp}

\DeclareUnicodeCharacter{2190}{\gets}

\DeclareUnicodeCharacter{2099}{_n}
\DeclareUnicodeCharacter{2098}{_m}
\DeclareUnicodeCharacter{1D5AD}{\ensuremath{\mathsf{N}}}

\DeclareUnicodeCharacter{1D4DF}{\mathcal{P}}

\newcommand\mydots{\makebox[0.6em][c]{.\hfil.\hfil.}}
\DeclareUnicodeCharacter{2026}{\mydots}
\DeclareUnicodeCharacter{226B}{\gg}
\DeclareUnicodeCharacter{2248}{\UnicodeApprox}
\DeclareUnicodeCharacter{227C}{\preceq}
\newcommand{\UnicodeApprox}{\ensuremath{\approx}}

\usepackage{stmaryrd}
\newcommand{\unicodeRelationalComposition}{\fatsemi}
\DeclareUnicodeCharacter{2A3E}{\unicodeRelationalComposition}
\usepackage{xcolor}

\definecolor{nordred}{HTML}{bf616a}
\definecolor{bordeaux}{HTML}{821529}
\definecolor{bluelink}{HTML}{003399}
\definecolor{nordred}{HTML}{bf616a}
\definecolor{nordblue}{HTML}{81a1c1}
\definecolor{norddarkblue}{HTML}{5e81ac}
\definecolor{nordgreen}{HTML}{a3be8c}
\definecolor{nordnight}{HTML}{4c566a}

\AtEndPreamble{
  \RequirePackage{hyperref}
  \RequirePackage{cleveref}
  \hypersetup{
    breaklinks = true,
    linktocpage,
    colorlinks = true,
    linkcolor = nordnight,
    citecolor = nordgreen,
    urlcolor = nordblue
  }
} %
\makeatletter
\newcommand{\nicelinktarget}[1]{\Hy@raisedlink{\hypertarget{#1}{}}}
\makeatother

\newcommand\defining[2]{\nicelinktarget{#1}{\color{black}{#2}}} %
\newcommand\Set{\hyperlink{linkSet}{\mathbf{Set}}}

\newcommand\stringDiagram{\hyperlink{linkPhysicalStringDiagram}{string diagram}}
\newcommand\stringDiagrams{\hyperlink{linkPhysicalStringDiagram}{string diagrams}}

\newcommand\physicalStringDiagrams{\hyperlink{linkPhysicalStringDiagram}{physical string diagrams}}

\newcommand\symmetricMonoidalCategories{\hyperlink{linkSymmetricMonoidal}{symmetric monoidal categories}}

\newcommand{\colorPre}[1]{{\color{nordred}{#1}}}
\newcommand{\colorMon}[1]{{\color{norddarkblue}{#1}}}

\newcommand\R{\colorPre{R}}

\newcommand\id{\mathrm{id}}

\newcommand{\physicalDuoidalCategory}{\hyperlink{linkPhysicalDuoidal}{physical duoidal category}}
\newcommand{\PhysicalDuoidalCategory}{\hyperlink{linkPhysicalDuoidal}{Physical duoidal category}}
\newcommand{\physicalDuoidalCategories}{\hyperlink{linkPhysicalDuoidal}{physical duoidal categories}}
\newcommand{\PhysicalDuoidalCategories}{\hyperlink{linkPhysicalDuoidal}{Physical duoidal categories}}

\newcommand{\hypergraph}{\hyperlink{linkHypergraph}{hypergraph}}

\newcommand{\hypergraphs}{\hyperlink{linkHypergraph}{hypergraphs}}

\newcommand{\zetlessPoset}{\hyperlink{linkZetlessPoset}{zetless poset}}

\newcommand{\zetlessPosets}{\hyperlink{linkZetlessPoset}{zetless posets}}
\newcommand{\ZetlessPosets}{\hyperlink{linkZetlessPoset}{Zetless posets}}

\newcommand{\poset}{\hyperlink{linkPoset}{poset}}

\newcommand{\posets}{\hyperlink{linkPoset}{posets}}

\newcommand{\incomparable}{\hyperlink{linkIncomparable}{incomparable}}

\newcommand{\incomparableConnectedness}{\hyperlink{linkIncomparableConnected}{incomparable connectedness}}

\newcommand{\incomparableConnected}{\hyperlink{linkIncomparableConnected}{incomparable connected}}

\newcommand{\parPrime}{\hyperlink{linkParPrime}{$⊗$-prime}}
\newcommand{\seqPrime}{\hyperlink{linkSeqPrime}{$⊲$-prime}}

\newcommand{\strictPhysicalDuoidalCategory}{\hyperlink{linkStrictPhysicalDuoidal}{strict physical duoidal category}}
\newcommand{\strictPhysicalDuoidalCategories}{\hyperlink{linkStrictPhysicalDuoidal}{strict physical duoidal categories}}

\newcommand{\StrictPhysicalDuoidalCategories}{\hyperlink{linkStrictPhysicalDuoidal}{Strict physical duoidal categories}}

\newcommand{\strictPhysicalDuoidalFunctor}{\hyperlink{linkStrictPhysicalDuoidalFunctor}{strict physical duoidal functor}}
\newcommand{\strictPhysicalDuoidalFunctors}{\hyperlink{linkStrictPhysicalDuoidalFunctor}{strict physical duoidal functors}}

\newcommand{\PhyDuo}{\hyperlink{linkPhyDuoCategory}{\ensuremath{\mathbf{PhyDuo}}}}
\newcommand{\PhySig}{\hyperlink{linkPhySigCategory}{\ensuremath{\mathbf{PhySig}}}}
\newcommand{\Zetless}{\hyperlink{linkZetlessFunctor}{\ensuremath{\mathsf{Zetless}}}}
\newcommand{\zetless}{\hyperlink{linkZetlessSet}{\ensuremath{\mathsf{zetless}}}}
\newcommand{\Obj}{\hyperlink{linkDuoidalObj}{\ensuremath{\mathsf{Obj}}}}

\newcommand{\physicalDuoidalSignature}{\hyperlink{linkPhysicalDuoidalSignature}{physical duoidal signature}}
\newcommand{\physicalDuoidalSignatures}{\hyperlink{linkPhysicalDuoidalSignature}{physical duoidal signatures}}

\newcommand{\PhysicalDuoidalSignatures}{\hyperlink{linkPhysicalDuoidalSignature}{Physical duoidal signatures}}

\newcommand{\signatureHomomorphism}{\hyperlink{linkSignatureHomomorphism}{signature homomorphism}}
\newcommand{\signatureHomomorphisms}{\hyperlink{linkSignatureHomomorphism}{signature homomorphisms}}

\newcommand{\interval}{\hyperlink{linkInterval}{interval}}
\newcommand{\intervals}{\hyperlink{linkInterval}{intervals}}

\newcommand{\subs}[3]{\ensuremath{#1[#2 \backslash #3]}}

\newcommand{\duoidalExpression}{\hyperlink{linkDuoidalExpression}{duoidal expression}}
\newcommand{\duoidalExpressions}{\hyperlink{linkDuoidalExpression}{duoidal expressions}}

\newcommand{\DuoidalExpressions}{\hyperlink{linkDuoidalExpression}{Duoidal expressions}}

\newcommand{\listType}{\hyperlink{linkListTypes}{\mathsf{listType}}}
\newcommand{\Expr}{\hyperlink{linkExpr}{\mathsf{expr}}}
\newcommand{\expr}{\hyperlink{linkExpr}{\mathsf{expr}}}
\newcommand{\lbl}{\mathsf{label}}
\newcommand{\Label}{\mathsf{label}}
\newcommand{\Encode}{\hyperlink{linkEncode}{\mathsf{encode}}}

\newcommand{\Input}{\hyperlink{linkInputOutputHypergraph}{\mathsf{input}}}
\newcommand{\Output}{\hyperlink{linkInputOutputHypergraph}{\mathsf{output}}}
\newcommand{\opt}{\Output}
\newcommand{\source}{\mathsf{source}}
\newcommand{\target}{\mathsf{target}}
\newcommand{\Source}{\mathsf{source}}
\newcommand{\Target}{\mathsf{target}}

\newcommand{\phyString}{\hyperlink{linkPhyString}{\mathsf{PhyString}}}
\newcommand{\Forget}{\hyperlink{linkForgetFunctor}{\mathsf{Forget}}}

\renewcommand{\UnicodeApprox}{\ensuremath{\hyperlink{linkUpToSymmetry}{\approx}}}

\newcommand{\physicalHypergraph}{\hyperlink{linkPhysicalHypergraph}{physical hypergraph}}
\newcommand{\physicalHypergraphs}{\hyperlink{linkPhysicalHypergraph}{physical hypergraphs}}

\newcommand{\Level}{\mathrm{Level}}
\newcommand{\Atomic}{\mathrm{Atomic}}
\newcommand{\Par}{\ensuremath{\mathrm{par}}}

\newcommand{\wires}{\mathsf{wires}}
\newcommand{\nodes}{\mathsf{nodes}}
\newcommand{\Wires}{\mathsf{wires}}
\newcommand{\Nodes}{\mathsf{nodes}} %
\newcommand{\SeeAppendix}[1]{\begin{proof}See Appendix, \Cref{#1}.\end{proof}}
 
\AtEndPreamble{
\theoremstyle{plain}
\newtheorem{theorem}{Theorem}[section] %
\newtheorem{proposition}[theorem]{Proposition}

\newtheorem{lemma}[theorem]{Lemma}
\newtheorem{corollary}[theorem]{Corollary}
\theoremstyle{definition}
\newtheorem{definition}[theorem]{Definition}

\theoremstyle{remark}
\newtheorem{remark}[theorem]{Remark}

\usepackage{tikz}
\usepackage{tikz-cd}
\usetikzlibrary{external}
 }

\allowdisplaybreaks

\usepackage{amsaddr}
\title[String Diagrams for Physical Duoidal Categories]{String Diagrams for Physical Duoidal Categories}
\author{Mario Román}

\makeatletter
\def\@copyrightspace{\relax}
\makeatother

\begin{document}
\begin{abstract}
  We introduce string diagrams for \physicalDuoidalCategories{} (normal symmetric duoidal categories): they consist of string diagrams with wires forming a zigzag-free partial order and order-preserving nodes whose inputs and outputs form intervals.
\end{abstract}
\keywords{Physical duoidal category, normal duoidal category, string diagrams.}
\maketitle

\section{Introduction}

\PhysicalDuoidalCategories{} (or normal $⊗$-symmetric duoidal categories) have been applied to the study of process dependencies \cite{shapiro2022duoidal, earnshaw2024produoidal}. We take this intuition seriously to develop a string diagrammatic calculus of \physicalDuoidalCategories{}.
String diagrams for \physicalDuoidalCategories{} particularize both to the hypergraph-based diagrams of symmetric monoidal categories \cite{joyal1991geometry} and to string diagrams for some spacial monoidal categories \cite{selinger2010survey}: essentially, they are string diagrams where wires form a \poset{} and nodes must take \intervals{} of the poset as inputs and outputs.
 
\begin{figure}[!ht]
  \centering
  \includegraphics[width=.35\textwidth]{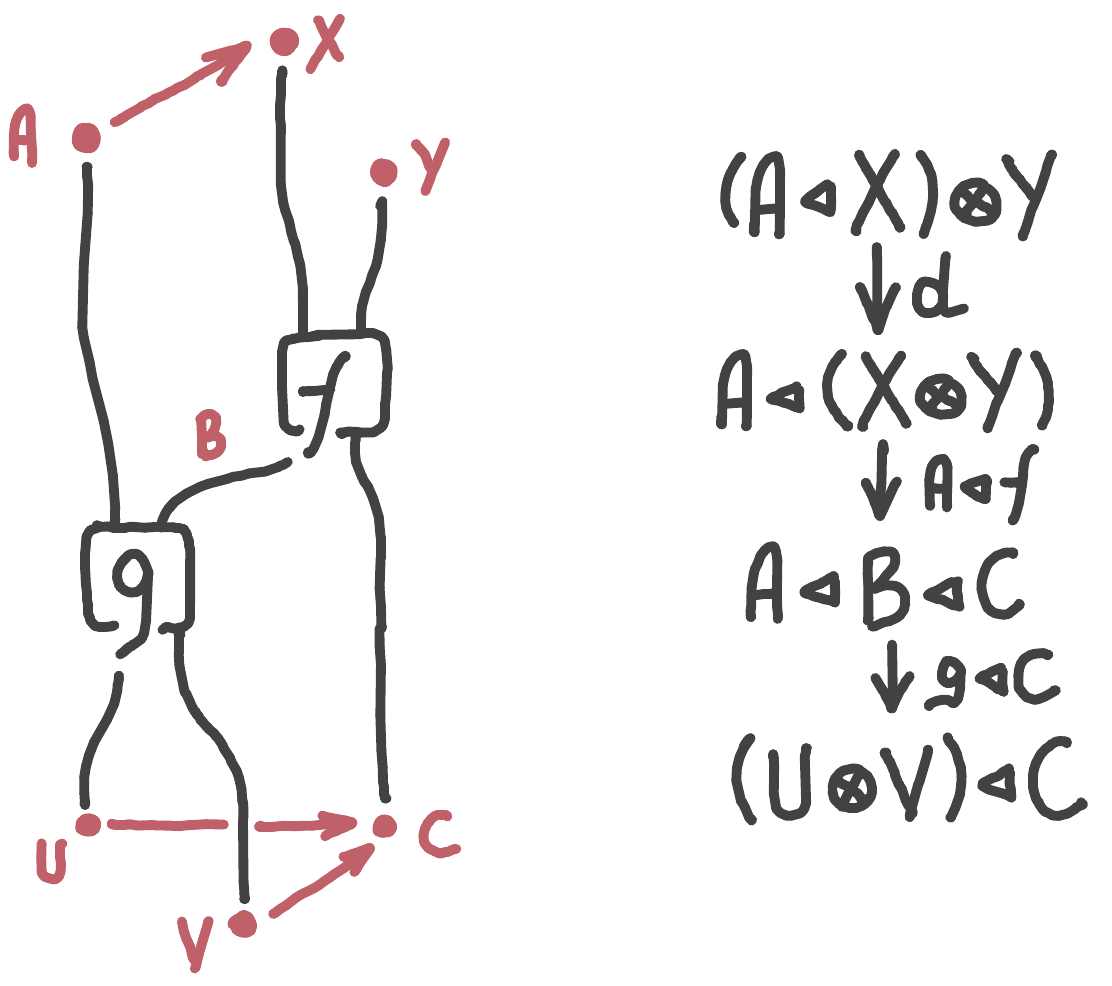}
  \caption{Example translation: from string diagrams to text.}
  \label{fig:exampleStrings}
\end{figure}

\begin{remark}
  Assume we want to compose two generators, $f ፡ X ⊗ Y → C$ and $g ፡ A ⊲ B → U ⊗ V$ into a morphism $(A ⊲ X) ⊗ Y → (U ⊗ V) ⊲ C$. We may first use the physical duoidal distributors to get a map $(A ⊲ X) ⊗ Y → A ⊲ (X ⊗ Y)$, and then apply $(A ⊲ f) ፡ A ⊲ (X ⊗ Y) → A ⊲ B ⊲ C$ and $(g ⊲ C) ፡ A  ⊲ B ⊲ C → (U ⊗ V) ⊲ C$.
  
  Our string diagram (in~\Cref{fig:exampleStrings}) starts with three wires: $a$, $x$ and $y$, assuming $a ≤ x$. The subposet $\{x, y\}$ forms an \interval{}, so we can apply $f$ to get the wires $b$ and $c$, which now must satisfy $b ≤ c$ (because of the type of $f$); while also satisfying $a ≤ b$ and $a ≤ c$ (because of $a ≤ x$). Again, $\{a,b\}$ form an \interval{}, and we can apply $g$ to get the wires $u$ and $v$, which are independent but both below $c$, leaving us with $u ≤ c$ and $v ≤ c$. We can either choose to keep track of these dependencies in our head or write them explicitly in the string diagrams: if we follow the typing rules, we need to declare no explicit dependencies, except for those of the input and output types.
\end{remark}

The main result of this paper (\Cref{thm:freeness}) is the construction of the free \physicalDuoidalCategory{} over a \physicalDuoidalSignature{}; the morphisms of this free \physicalDuoidalCategory{} consist of \stringDiagrams{} with strings ordered by a poset without zigzags. In other words, we construct an adjunction between the category of \strictPhysicalDuoidalCategories{}, $\PhyDuo$, and the category of \physicalDuoidalSignatures{}, $\PhySig{}$.

\PhysicalDuoidalCategories{}, because of their relation to \posets{}, have multiple potential applications to causality, concurrency, and formal category theory. However, without an appropriate syntax, reasoning with them can be tedious and unenlightening. The string diagrammatic syntax may dramatically simplify proofs regarding duoidal structures.

\section{Physical Duoidal Categories}

\PhysicalDuoidalCategory{} is the term Shapiro and Spivak \cite{shapiro2022duoidal} give to normal $⊗$-symmetric duoidal categories \cite{aguiar2010monoidal,garner2016commutativity}.

\begin{definition}[Strict physical duoidal category]
  \defining{linkStrictPhysicalDuoidal}{}
  \defining{linkPhysicalDuoidal}{}
  A \emph{strict physical duoidal category} is a category with a strict monoidal structure and a strict symmetric monoidal structure sharing the same unit, $(𝕍,⊲,⊗,N)$, and such that the first monoidal structure distributes over the second; that is, there exist maps
  \begin{align*}
    & 𝒹_{X,Y,Z,W} ፡ (X ⊲ Z) ⊗ (Y ⊲ W) → (X ⊗ Y) ⊲ (Z ⊗ W); \\
    & 𝓈_{X,Y} ፡ X ⊗ Y → Y ⊗ X.
  \end{align*}
  \StrictPhysicalDuoidalCategories{} are defined to be coherent structures, meaning that any formally distinctly typed equation of morphisms on the free \strictPhysicalDuoidalCategory{} holds true.
\end{definition}

\begin{definition}[Strict physical duoidal functor]
  \defining{linkStrictPhysicalDuoidalFunctor}
  A \emph{strict physical duoidal functor} between two \strictPhysicalDuoidalCategories{}, 
  $$(𝕍,⊗_V,⊲_V,N_V,𝒹_V,𝓈_V)\mbox{ and }(𝕎,⊗_W,⊲_W,N_W, 𝒹_W, 𝓈_W),$$ is a functor that is strict symmetric monoidal for the parallel monoidal structures, $(⊗_V)$ and $(⊗_W)$, and strict monoidal for the sequential monoidal structures, $(⊲_V)$ and $(⊗_W)$; and that, moreover, strictly preserves the structure maps, $F(𝒹_V) = 𝒹_W$ and $F(𝓈_V) = 𝓈_W$.
\end{definition}

\begin{proposition}
  \label{prop:PhyDuoCategory}
  \defining{linkPhyDuoCategory}{}
  \StrictPhysicalDuoidalCategories{} and \strictPhysicalDuoidalFunctors{} between them form a category, $\PhyDuo$.
\end{proposition}
\begin{proof}
  See Appendix, \Cref{ax:prop:PhyDuoCategory}.
\end{proof}

\section{Duoidal expressions}
Let us start by studying duoidal expressions – the objects of the free \physicalDuoidalCategory{} over a set of objects. These form a ``boring'' construction of the objects of the free \physicalDuoidalCategory{}: if we care about string diagrams, the interesting bit of mathematics will be in the construction of the morphisms of the free \physicalDuoidalCategory{}.

\begin{definition}[Duoidal expression]
  \defining{linkDuoidalExpression}{}
  \defining{linkExpr}{}
  The set of \emph{duoidal expressions}, $\Expr(A)$, over some set of objects $A$ is inductively built as
  \begin{itemize}
    \item the empty expression, $N$;
    \item a singleton $a$, for each $a ∈ A$;
    \item a sequence of expressions, $E = E₁ ⊲ ... ⊲ Eₙ$, each not empty nor sequencing;
    \item a tensoring of expressions, $E = E₁ ⊗ ... ⊗ Eₙ$, each not empty nor tensoring.
  \end{itemize}
  In other words, by definition, we forbid unncessary nested expressions, such as $(A ⊗ B) ⊗ C$; we only allow the correspoding reduced expressions, such as $A ⊗ B ⊗ C$. This allows us to avoid redundancy and construct a \strictPhysicalDuoidalCategory{}; however, this also forces us to define non-trivial operations for substitution, parallel, and tensoring compositions (see \Cref{prop:seqpar-exp}).
\end{definition}

\begin{proposition}[Sequencing and tensoring duoidal expressions]
  \label{prop:seqpar-exp}
  There exist two binary operations in \duoidalExpressions{}, $(⊲_e) ፡ \Expr(A) × \Expr(A) → \Expr(A)$ and $(⊗_e) ፡ \Expr(A) × \Expr(A) → \Expr(A)$, defined by sequencing and tensoring after reducing by associativity and unitality.
\end{proposition}

\begin{proposition}
  \DuoidalExpressions{} induce a monad $\Expr ፡ \Set → \Set$.  The objects of any \strictPhysicalDuoidalCategory{} form an algebra for the monad, $$⟦•⟧ ፡ \Expr(𝕍_{obj}) → 𝕍_{obj}.$$
\end{proposition}

\begin{definition}[Equality up to symmetry]
  \defining{linkUpToSymmetry}{}
  Two \duoidalExpressions{} are \emph{equal up to $⊗$-symmetry} (or, simply, \emph{up to symmetry}) if they are related by the following inductively defined relation $(≈) ⊆ \Expr(A) × \Expr(A)$.
  \begin{itemize}
    \item $N ≈ N$;
    \item $a ≈ a$, for each $a ∈ A$;
    \item $E₁ ⊲ ... ⊲ Eₙ ≈ E₁' ⊲ ... ⊲ Eₙ'$, for $Eᵢ ≈ Eᵢ'$;
    \item $E₁ ⊗ ... ⊗ Eₙ ≈ E'_{σ(1)} ⊗ ... ⊗ E'_{σ(n)}$, for $Eᵢ ≈ Eᵢ'$ permuted by $σ ∈ \mathrm{Perm}(n)$.
  \end{itemize}
\end{definition}

\section{Posets and Zetless Posets}
Let us give some basic definitions on \posets{}. The rest of the text will revolve around those \posets{} that contain no zigzags (the \emph{\zetlessPosets{}}), but most operations and definitions apply to arbitrary \posets{}.

\begin{definition}[Poset]
  \defining{linkPoset}
  A (finite) \emph{poset} is a finite set endowed with a reflexive, antisymmetric, and transitive relation, $(≤)$, on its elements. 
\end{definition}

\begin{remark}
  Given any relation on a finite set, $(→)$, its reflexive, antisymmetric, and transitive closure forms a \poset{}.
When we draw \posets{}, we use arrows $(→)$ instead of the less-or-equal-than symbols $(≤)$ to represent the generators of the \poset{}: the \poset{} is the closure under transitivity and reflexivity (quotiented by antisymmetry if necessary) of these generators. 
\end{remark}

\begin{definition}[Incomparability]
  \defining{linkIncomparable}
  Two elements on a \poset{}, $x,y ∈ P$, are \emph{incomparable}, $x ∥ y$, if neither $x ≤ y$ nor $y ≤ x$ are true. %
\end{definition}

\begin{definition}[Incomparable connectedness]
  \defining{linkIncomparableConnected}
  Two elements, $x, y ∈ P$, are \emph{incomparable connected} if there exists a path of pairwise \incomparable{} elements between them, $x ∥ p₁$, $p₁ ∥ p₂$, ..., $pₙ ∥ y$. An \emph{incomparable connected component} is a full subposet such that all objects are \incomparableConnected{} between them.
\end{definition}

\begin{definition}[Sequencing of posets]
  \defining{linkPosetSequencing}{}
  The \emph{sequencing} of two \posets{}, $P$ and $Q$, is the \poset{} that contains the disjoint union of the objects of both \posets{}, all the edges of both $P$ and $Q$, and an edge from every element of $P$ to every element of $Q$. That is,
  $$P ⊲ Q = (P_{obj} + Q_{obj}; ≤_{P} + ≤_{Q} + \{p ≤ q\}_{p ∈ P, q ∈ Q}).$$ 
\end{definition}

\begin{definition}[Tensoring of posets]
  \defining{linkPosetTensoring}{}
  The tensoring of two \posets{}, $P$ and $Q$, is the \poset{} containing the disjoint union of objects from both, and the disjoint union of edges from both.
  $$P ⊗ Q = (P_{obj} + Q_{obj}; ≤_{P} + ≤_{Q}).$$
\end{definition}

\subsection{Zetless posets}
\ZetlessPosets{} are \posets{} without zigzags inside them.
We note that both the empty \poset{} and the singleton \poset{} are \zetlessPosets{}; the tensoring and sequencing of \zetlessPosets{} also form \zetlessPosets{}. 

\begin{definition}[Zetless poset]
  \defining{linkZetlessPoset}{}
  A \emph{zetless poset} is a \poset{} that does not admit a fully faithful embedding of the $\mathsf{Z}$-poset, $x → u ← y → v$.
\end{definition}

\begin{proposition}
  \label{prop:sequencingTensoringZetless}
  The sequencing and tensoring of two \zetlessPosets{} is again a \zetlessPoset{}.
\end{proposition}
\SeeAppendix{ax:prop:sequencingTensoringZetless}

\begin{proposition}
  \label{prop:connectedBySpan}
  In a \zetlessPoset{}, any two connected elements must be connected by either a span or a cospan. That is, if there is a path between two elements, $x₀ → x₁ ← x₂ → ... ← xₙ$, there must exist either a cospan between them, $x₀ → u ← xₙ$, or a span between them, $x₀ ← v → xₙ$.
\end{proposition}
\begin{proof}
  See Appendix, \Cref{ax:prop:connectedBySpan}.
\end{proof}

\subsection{Prime posets}
Primality for sequencing and primality for tensoring can be characterized by more familiar notions: connectedness and \incomparableConnectedness{}. We comment on this characterization, as it will simplify the rest of our proofs.

\begin{definition}[Prime posets]
  \defining{linkSeqPrime}{}
  \defining{linkParPrime}{}
  A non-empty \poset{} $P$ is parallel prime (or \parPrime{}) if $P = Q₁ ⊗ ... ⊗ Qₙ$ with $Qᵢ ≠ N$ implies $n = 1$ and $Q₁ = P$; it is sequential prime (or \seqPrime{}) if 
  $P = Q₁ ⊲ ... ⊲ Qₙ$ with $Qᵢ ≠ N$ implies $n = 1$ and $Q₁ = P$.
\end{definition}

\begin{proposition}
  \label{prop:parprime-connected} \label{prop:seqprime-incomparable} \label{prop:prime-prime-singleton}
  A \poset{} is \parPrime{} if and only if it is connected.
  A \poset{} $P$ is \seqPrime{} if and only if it is \incomparableConnected{}.
  Any \seqPrime{} and \parPrime{} \poset{} must be a singleton poset.
\end{proposition}
\SeeAppendix{ax:prop:parprime-connected,ax:prop:seqprime-incomparable,ax:prop:prime-prime-singleton}

\section{Posets versus Duoidal Expressions}

\begin{definition}[Typed zetless poset]
  \defining{linkZetlesSet}{}
  A \emph{zetless poset typed by a set $T$} is a \zetlessPoset{} structure whose set is finite, $\{1,...,n\}$, endowed with a type-labelling function $t ፡ \{1,...,n\} → T$, and considered up to type-preserving bijection. The set of \zetlessPosets{} labelled by a set $A$ is written as $\zetless(A)$.
\end{definition}
For instance, there are $7$ posets of cardinality $2$ typed over $\{A,B\}$ (\Cref{fig:ex:countposets}).
\begin{figure}[ht]
  \centering
  \includegraphics[width=.4\textwidth]{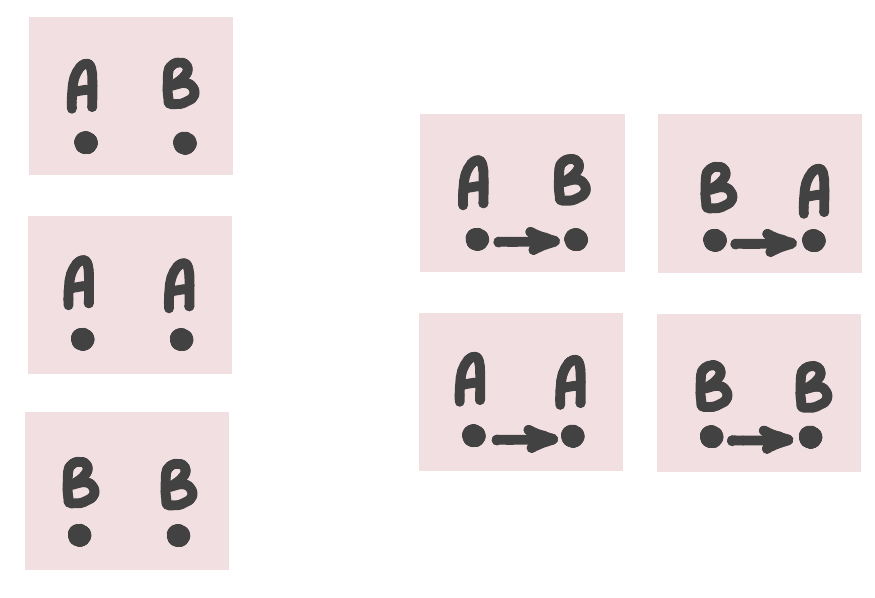}
  \caption{Zetless posets typed over the set $\{A,B\}$ with cardinality $2$.}
  \label{fig:ex:countposets}
\end{figure}

\begin{proposition}[Posets versus duoidal expressions]
  \defining{linkEncode}{}
  \label{prop:zetlessPosetsDuoidalExpressions}
  \ZetlessPosets{} labelled over a set are in correspondence with \duoidalExpressions{} on that set, up to symmetries of the tensored components,
  \[
    \zetless(A) ≅ \Expr(A) / (≈).
  \]
  In particular, there is a surjective function $\Encode ፡ \Expr(A) → \zetless(A)$.
\end{proposition}
\begin{proof}
  See Appendix, \Cref{ax:prop:zetlessPosetsDuoidalExpressions}.
\end{proof}

\begin{remark}
  In other words, \zetlessPosets{} are a symmetry-aware encoding of \duoidalExpressions{}. Every time we write a duoidal expression we must choose an order in which we write the \poset{}; however, $P ⊗ Q$ and $Q ⊗ P$ represent the same \poset{}.
\end{remark}

\begin{proposition}[Shapiro and Spivak {{\cite{shapiro2022duoidal}}}]
  \label{prop:inclusionmaps}
  The existence of an inclusion of \zetlessPosets{} corresponds to the existence of a structure map between their corresponding \duoidalExpressions{} in a \physicalDuoidalCategory{}.
\end{proposition}
\begin{proof}
  See Appendix, \Cref{ax:prop:inclusionmaps}.
\end{proof}

\begin{definition}
  \defining{linkZetlessFunctor}{}
  The category of \emph{zetless maps} over a set, $\Zetless(A)$, has, as objects, the \duoidalExpressions{}. The morphisms between two \duoidalExpressions{} are type-preserving bijective-on-objects inclusions between their corresponding \zetlessPosets{}.
\end{definition}

\begin{proposition}
  The construction of the category of zetless maps over a set induces a functor $\Zetless ፡ \mathbf{Set} → \PhyDuo$. The functor carries a function $f ፡ A → B$ into the functor that changes the types of the \duoidalExpressions{} on objects and transports the bijective-on-objects inclusions.
\end{proposition}

\begin{theorem}
  \label{thm:freePhysicalObjects}
  \defining{linkDuoidalObj}{}
  Zetless maps construct the free \physicalDuoidalCategory{} over a set of objects. In other words, the functor $\Zetless{} ፡ \mathbf{Set} → \PhyDuo{}$ is left adjoint to the forgetful functor that picks the objects of a \physicalDuoidalCategory{}, $\mathsf{Obj} ፡ \PhyDuo{} → \mathbf{Set}$.
\end{theorem}
\begin{proof}
  See Appendix, \Cref{ax:thm:freePhysicalObjects}.
\end{proof}

\begin{corollary}
  Unlabelled \duoidalExpressions{} and bijective-on-objects inclusions of unlabelled \zetlessPosets{} form the free \physicalDuoidalCategory{} over the singleton set.
\end{corollary}

\section{Intervals}

This first section has recalled the construction of the free \physicalDuoidalCategory{} on a set as a category of \zetlessPosets{} and bijective-on-object inclusions between them. String diagrams will need more than that: instead of constructing the free \physicalDuoidalCategory{} \emph{over a set}, we must construct it over an arbitrary signature of physical duoidal operations. This is what we will pursue for the rest of this text.

For this purpose, it becomes relevant to analyze the notion of subterm, and to characterize it in terms of \zetlessPosets{}. This is what we do now.

\begin{definition}[Interval]
  \defining{linkInterval}{}
  An \emph{interval} is a subset of a \poset{} that is closed under intermediate elements.
  That is, a subset of a poset $I \subseteq P$, is an interval if any element $y ∈ P$ in between two elements of the subset, $x₀ ≤ y ≤ x₁$ for $x₀, x₁ ∈ I$, belongs to the interval, $y ∈ I$.
\end{definition}

\begin{definition}[Substitution]
  The \emph{substitution} of a \poset{} $P$ into a \poset{} $Q$ at an element $x ∈ Q$ is a \poset{}, $\subs{Q}{x}{P}$, containing the objects of $P$ and $Q$ but excluding $x ∈ Q$, and containing all the edges of $P$, of $Q$, and from and to $x$.
  \begin{align*}
    \subs{Q}{x}{P} = (P_{obj} &+ Q_{obj} - \{x\} ; 
     ≤_{Q - \{x\}} + ≤_P + \\ & \{q ≤ p \mid q ∈ Q, q ≤ x\} +
    \{p ≤ q \mid q ∈ Q, x ≤ q\}).
  \end{align*}
\end{definition}

\begin{proposition}[Posets form an operad]
  Substituting two posets in two different elements yields the same result independently of the order of substitution,
  \[
    \subs{\subs{Q}{x}{P₁}}{y}{P₂} 
    = 
    \subs{\subs{Q}{y}{P₂}}{x}{P₁}.
  \]
  Substitution is associative, meaning that, substituting into a poset used for substitution at some element $y ∈ P₁$ is the same as substituting into that element on the resulting poset,
  \[
    \subs{\subs{Q}{x}{P₁}}{y}{P₂} 
    = 
    \subs{Q}{x}{\subs{P₁}{y}{P₂}}.
  \]
  \ZetlessPosets{} with substitution form an operad.
\end{proposition}

\begin{definition}[Bracketed]
  A full subposet $P \subseteq Q$ is \emph{bracketed} if any element $q ∈ Q$ above an element of the subposet, $p₀ ≤ q$ for $p₀ ∈ P$, is above all elements of the poset, $p ≤ q$ for any $p ∈ P$; and if any element $q ∈ Q$ below an element of the subposet, $q ≤ p₀$ for $p₀ ∈ P$, is above all elements of the poset, $q ≤ p$ for any $p ∈ P$.
\end{definition}

\begin{proposition}[Bracketed only if substituted]
  \label{prop:bracketedSubstituted}
  A \poset{} $R$ arises as a substitution $R = \subs{Q}{x}{P}$ of any of its full subposets, $P ⊆ R$, if and only if it is bracketed.
\end{proposition}
\SeeAppendix{ax:prop:bracketedSubstituted}

\begin{proposition}[Bracketed implies interval]
  Any bracketed \poset{} is an \interval{}.
\end{proposition}

\begin{proposition}[Interval if and only if bracketed in a saturation]
  \label{prop:subsetBracketed}
  A subset of a \zetlessPoset{} is an \interval{} if and only if it appears as a bracketed poset in some saturation of the \poset{}.
\end{proposition}
\begin{proof}
  See Appendix, \Cref{ax:prop:subsetBracketed}.
\end{proof}

\section{String Diagrams}

\subsection{Signatures}
The signature for a \stringDiagram{} for \physicalDuoidalCategories{} consists of a set of basic types and some generators that are typed by \duoidalExpressions{} on both the input and the outputs.

\begin{definition}[Physical duoidal signature]
  \defining{linkPhysicalDuoidalSignature}
  A \emph{physical duoidal signature} $𝓖$ is given by a set of basic types, $𝓖_{t}$, and a set of generators, $𝓖(Eᵢ; E_o)$, for each two \duoidalExpressions{} over the types, $Eᵢ, E_o ∈ \Expr(𝓖_{t})$.

  We write $𝓖$ for the set of all generators: there exist functions $\source ፡ 𝓖 → \Expr(𝓖ₜ)$ and $\target ፡ 𝓖 → \Expr(𝓖ₜ)$ picking the source and target \duoidalExpressions{} of the generator.
\end{definition}

\begin{definition}[Homomorphism of physical duoidal signatures]
  \defining{linkSignatureHomomorphism}{}
  Let $𝓖$ and $𝓗$ be two \physicalDuoidalSignatures{}. A \emph{homomorphism of physical duoidal signatures} (or, more succintly, a \emph{signature homomorphism}), $f ፡ 𝓖 → 𝓗$, consists of a function between basic types, $f_t ፡ 𝓖_{t} → 𝓗_{t}$, and a function of generators,
  \[
    f ፡ 𝓖(U;V) → 𝓗(\Expr(fₜ)(U); \Expr(fₜ)(V)).
  \]
\end{definition}

\begin{proposition}[Category of physical duoidal signatures]
  \defining{linkPhySigCategory}{}
  \PhysicalDuoidalSignatures{} and \signatureHomomorphisms{} between them form a category, $\PhySig$.
\end{proposition}

\begin{proposition}[Forgetful functor to physical duoidal signatures]
  \label{prop:forgetFunctor}
  \defining{linkForgetFunctor}{}
  There is a forgetful functor from the category of \strictPhysicalDuoidalCategories{} to the category of \physicalDuoidalSignatures{}, 
  $$\mathsf{Forget} ፡ \PhyDuo → \PhySig.$$
  The forgetful functor picks all of the objects of a category as basic types, $\mathsf{Forget}(𝕍)_{t} = 𝕍_{obj}$; it picks all of the morphisms of a given type as generators,
  \[
    \Forget(𝕍)(U;V) = 𝕍(⟦U⟧;⟦V⟧).
  \]
\end{proposition}
\SeeAppendix{ax:prop:forgetPhysical}

\subsection{Hypergraphs and String Diagrams}
The combinatorial structure behind \physicalStringDiagrams{} is that of linear and acyclic \hypergraphs{}: the \hypergraph{} contains wires and nodes, and each wire connects the output of exactly one node to the input of exactly one node. The main result of this section will show that linear acyclic hypergraphs whose wires are ordered following certain rules form the free \physicalDuoidalCategory{} over a \physicalDuoidalSignature{}. 

Of course, there exists a more obvious construction of the free \physicalDuoidalCategory{}: we can inductively write all possible terms arising from composition and tensoring and then quotient by the appropriate equations. However, the interesting bit of mathematics is to show that \physicalStringDiagrams{} are just as good as terms for that purpose.

Finally, note that the structure we can extract from a diagram is the connectivity – the hypergraph – and not how this connectivity has been drawn.  The situation is similar (if not identical) to the hypergraph representation for strict symmetric monoidal categories \cite{bonchi2019diagrammatic}.

\begin{definition}[Hypergraph]
  \defining{linkHypergraph}{}
  \defining{linkInputOutputHypergraph}{}
  A \emph{hypergraph} $H$ consists of a finite set of wires (or hypervertices), $\wires(H)$, and a finite set of nodes (or posetal edges), $\nodes(H)$, with a source and target functions, 
  $$\Input ፡ \nodes(H) → \mathrm{List}(\wires(H))\ \mbox{ and }\ \Output ፡ \nodes(H) → \mathrm{List}(\wires(H)).$$
  Hypergraphs are considered equal up to source-and-target-preserving isomorphism of their sets of wires and nodes.
\end{definition}

\begin{definition}[Wire-linear hypergraph]
  A \hypergraph{} $H$ is \emph{wire-linear} if every wire $w ∈ H_W$ appears exactly once as a source and once as a target. That is, there uniquely exist two nodes $s^{\ast}(w)$ and $t^{\ast}(w)$ such that $t(s^{\ast}(w)) = Γ,w,Γ'$ and $s(t^{\ast}(w)) = Δ,w,Δ'$, with $w \not\in Γ,Γ',Δ,Δ'$. These induce functions $s^{\ast} ፡ H_W → H_N$ and $t^{\ast} ፡ H_W → H_N$.
\end{definition}

\begin{definition}[Acyclic hypergraph]
  A \hypergraph{} $H$ is \emph{acyclic} if it contains no closed paths of non-zero length. A path of length $n ∈ ℕ$ is a sequence of wires $w₀,...,wₙ ∈ H_W$ such that the target of a wire is the source of the next wire, $t^{\ast}(w_i) = s^{\ast}(w_{i+1})$ for each $i = 0,...,n-1$. A path is closed whenever $w₀ = wₙ$.
\end{definition}

\begin{definition}[List of types]
  \defining{linkListTypes}{}
  The \emph{list of types} of a \duoidalExpression{} is defined inductively to match the list of types we would obtain traversing the expression from left to right. Explicitly, it is defined inductively as
  \begin{itemize}
    \item $\listType(N) = []$;
    \item $\listType(a) = [a]$;
    \item $\listType(E₁ ⊲ ... ⊲ Eₙ) = \listType(E₁), ..., \listType(Eₙ)$;
    \item $\listType(E₁ ⊗ ... ⊗ Eₙ) = \listType(E₁), ..., \listType(Eₙ)$.
  \end{itemize}
\end{definition}

\begin{definition}[Physical hypergraph]
  \defining{linkPhysicalHypergraph}{}
  \label{def:physicalHypergraph}
  A \emph{physical hypergraph} $H$ labelled over a \physicalDuoidalSignature{} $𝓖$ is a wire-linear acyclic \hypergraph{} that admits a poset,
  $$(⊑) \subseteq \wires(H)  × \wires(H),$$
  such that the following properties hold.
  \begin{enumerate}
    \item Wires, $w ∈ \Wires(H)$, are labelled by basic types, $\lbl(w) ∈ 𝓖_{obj}$.
    \item Nodes, $n ∈ \Nodes(H)$, are labelled by generators, $\lbl(n) ∈ 𝓖$.
    \item For each node $n ∈ H_N$, input wires are typed by the input expression; output wires are typed by the output expression, 
    \begin{align*}
      \lbl(\Input(n)) &= \listType(\Source(\lbl(n))), \\ 
      \lbl(\opt(n)) &= \listType(\Target(\lbl(n))).
    \end{align*}
    \item For each node $n ∈ H_N$, the induced poset over the input wires, $\Input(n)$, forms an interval. There must exist an identity-on-objects inclusion of the input wires into the poset of inputs of the generator, $$\Input(n) ⭇ \Source(\Label(n)).$$
    \item The induced poset over the output wires coincides with the target \zetlessPoset{} of the generator, $\Output(n) = \Target(\Label(n))$.
    \item Input wires induce order to the output wires. That is, for any node $n ∈ \Nodes(H)$ and any of its outputs $o ∈ \Output(n)$, any wire is above it $x ⊑ o$ if and only if there exists an input $i ∈ \Input(n)$ above the wire, $x ⊑ i$; any output is below a wire, $o ⊑ x$, if and only if there exists an input $i ∈ \Input(n)$ below the wire, $i ⊑ x$.
  \end{enumerate}
\end{definition}

Note that this last condition never breaks asymmetry because the input is an interval: if there were wire below and above the input, then it should belong to the input.  

\begin{definition}[Physical string diagram]
  \defining{linkPhysicalStringDiagram}{}
  A \emph{physical string diagram}, $α ፡ E₁ → E₂$, labelled on a \physicalDuoidalSignature{} $𝓖$ and from a \duoidalExpression{} $E₁ ∈ \expr(𝓖ₜ)$ to a \duoidalExpression{} $E₂  ∈ \expr(𝓖ₜ)$, is a  \physicalHypergraph{}, $α$, endowed with special unlabelled input and output nodes, $i, o ∈ \nodes(α)$, satisfying $\Label(\Output(i)) = E₁$ and $\Label(\Input(o)) = E₂$, while $\Label(\Input(i)) = \Label(\Output(o)) = N$.
\end{definition}

\begin{remark} 
  Every node is labelled by a \duoidalExpression{} for its input and its output. The type list of that expression fixes a linear ordering to the inputs and the outputs. When we draw, the inputs and outputs are ordered from left to right, and thus they are identified. 
\end{remark}

\begin{remark}
  We decide that the wires will always keep the minimal possible \poset{} that is compatible with the diagram. This forces an asymmetry between rules (4) and (5); however, this is merely a convention: we could have also decided that the wires always keep the maximal possible poset that is compatible with the diagram.
\end{remark}

\begin{theorem}[Diagrams form a physical duoidal category]
  \label{th:diagramsPhysical}
    \defining{linkPhyString}{}
    Physical duoidal \stringDiagrams{} over a \physicalDuoidalSignature{}, $𝓖$ form a \physicalDuoidalCategory{}, $\mathsf{PhyString}(𝓖)$.
  
    Objects are \duoidalExpressions{} over the types of the \physicalDuoidalSignature{}; morphisms are hypergraphs with input and output given by the \zetlessPosets{} corresponding to the \duoidalExpressions{}.
    \begin{figure}[ht]
      \centering
      \includegraphics[width=.8\textwidth]{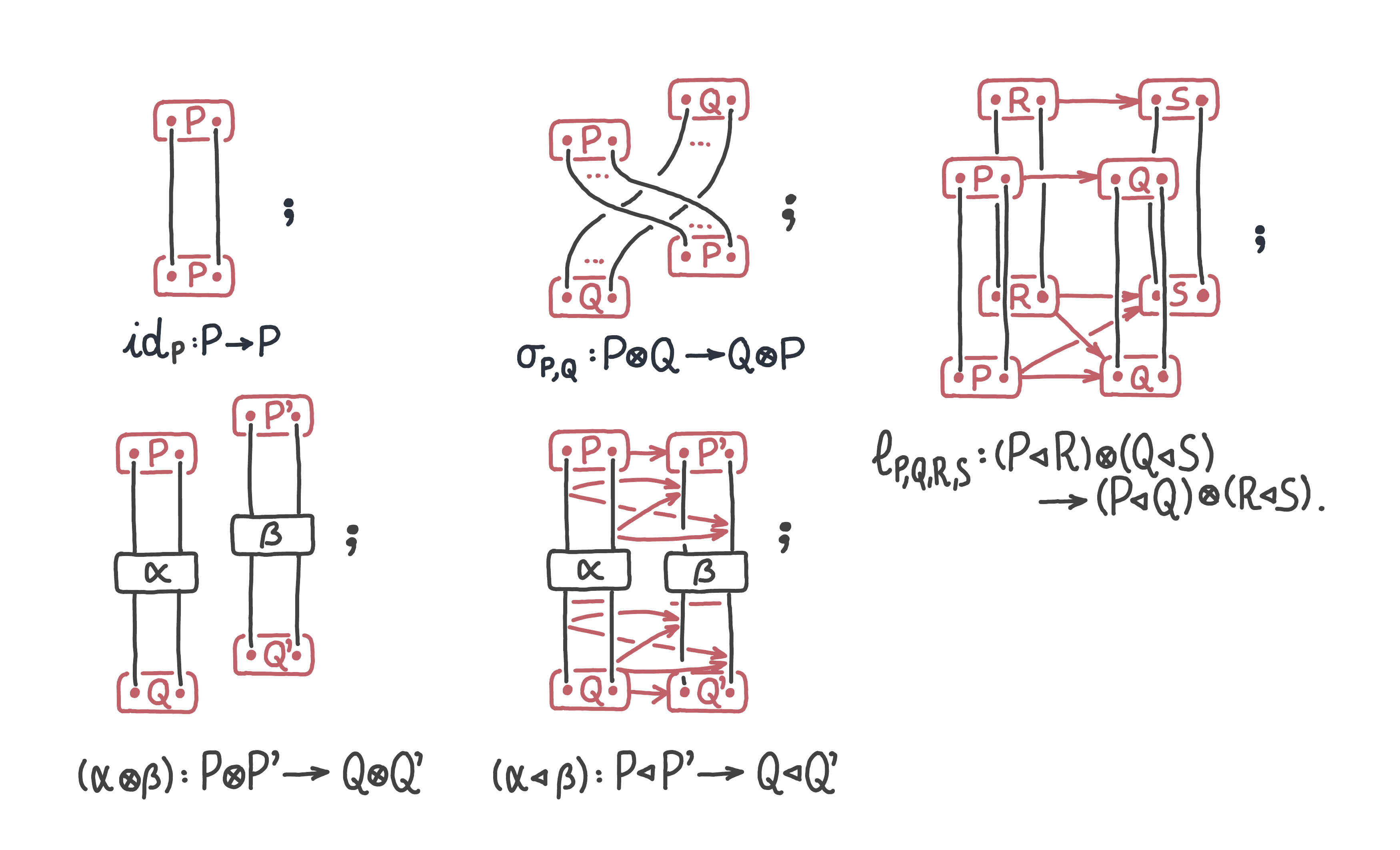}
      \caption{Physical string diagrams form a physical duoidal category.}
      \label{fig:phystring:phycategory}
    \end{figure}
  
    Composing two string diagrams, $α$ and $β$, concatenates them, so that the output wires of $α$ and the input wires of $β$ get merged into single wires. Parallel tensoring juxtaposes diagrams; sequential tensoring juxtaposes but also links every wire from the first diagram to the second diagram (see \Cref{fig:phystring:phycategory}).
\end{theorem}
\SeeAppendix{ax:th:diagramsPhysical}

\begin{proposition}[String diagrams functor]
  \label{prop:stringDiagramsFunctor}
  The construction of physical duoidal string diagrams over a \physicalDuoidalSignature{} extends to a functor $$\phyString ፡ \PhySig → \PhyDuo.$$  
\end{proposition}
\SeeAppendix{ax:prop:stringDiagramsFunctor}

\subsection{Freeness}
We conclude this text by proving that \stringDiagrams{} over a \physicalDuoidalCategory{} form the free \physicalDuoidalCategory{} over a \physicalDuoidalSignature{}. The idea of the proof is that every hypergraph can be decomposed in multiple \emph{atomic hypergraphs}, consisting of a single node and wires. The interpretation of each one of these atomic hypergraphs is completely determined; thus, we are forced to conclude that the interpretation of each hypergraph is completely determined. This decomposition into atomic hypergraphs will not be unique, but the interpretation will still be well-defined thanks to the interchange law for duoidal categories.

\begin{definition}[Nodes connected by a wire]
  Two nodes of a \hypergraph{}, $n₁, n₂ ∈ H_N$ are \emph{connected by a wire}, $n₁ ≼_0 n₂$, when there exists a wire, $w ∈ H_W$, going from the output of the first node, $w ∈ \Output(n₁)$ to the input of the second node, $w ∈ \Input(n₂)$.  This defines a relation, $(≼_0) ፡ H_N × H_N$.
\end{definition}

\begin{definition}[Ordering the nodes]
  \label{def:orderingNodes}
  Let $H$ be a physical string diagram. An \emph{ordering} on the nodes is any total order, $(≼) ፡ H_N × H_N$, that extends the relation being connected by a wire, $(≼_0) ፡ H_N × H_N$.
\end{definition}

Note that this ordering always exist: \physicalHypergraphs{} are acyclic. However, it is not unique: $(≼_0)$ is typically not a total order.

\begin{lemma}[Parallel wires]
  For any string diagram $H$ and any ordering on its nodes $(≼)$, we can define the set of wires \emph{parallel to a node}, $\Par(n)$, to contain exactly those wires $w ∈ W$ that are neither on the input below the node – $w \notin \Input(m)$ for each $m ≼ n$ – nor on the output above the node – $w \notin \Output(m)$ for each $n ≼ m$. Parallel wires, $\Par(n)$, together with the node, $n$, form a poset, $\mathrm{Level}(n)$.   
\end{lemma}

\begin{definition}[Atomic hypergraph]
  The \emph{atomic hypergraph} of a node $n$ has a single node and three types of wires, $H_W = \Input(n) + \Output(n) + \Par(n)$: (1) those on the input of the node; (2) those on the output of the node; and (3) those parallel to the node.
  $$\Atomic(n) ፡ \Level(n)[n \backslash \Input(n)] → \Level(n)[n \backslash \Output(n)]$$
\end{definition}

\begin{proposition}[Decomposition into atomic hypergraphs]
  \label{prop:physicalHypergraphAtomic}
  Every \physicalHypergraph{} can be written as the composition of atomic hypergraphs.
\end{proposition}
\SeeAppendix{ax:prop:physicalHypergraphAtomic}

\begin{lemma}[String diagram universal map]
  \label{lemma:universalPhyString}
  For each \physicalDuoidalSignature{}, $𝓖$, there exists a \signatureHomomorphism{}, $u_{𝓖} ፡ 𝓖 → \Forget(\phyString(𝓖))$.
\end{lemma}
\SeeAppendix{ax:lemma:universalPhyString}

\begin{theorem}[String diagram adjunction]
  \label{thm:freeness}
  There exists an adjunction from the category of \physicalDuoidalSignatures{} to the category of \strictPhysicalDuoidalCategories{} given by physical duoidal string diagrams $\phyString ፡ \PhySig → \PhyDuo$ and the forgetful functor, $\mathsf{forget} ፡ \PhyDuo → \PhySig$.
\end{theorem}
\SeeAppendix{ax:thm:freeness}

\section{Examples}

Let us study multiple examples of constructions in \physicalDuoidalCategories{} using the string diagrams that we formalize in the rest of the text.

\subsection{Duoids}
In the same way that the \emph{microcosm principle} prescribes that monoids are more generally defined in a multicategory and that Frobenius monoids are more generally definable on a polycategory, duoids are more generally defined in a duoidal category.

\begin{figure}[ht]
  \centering
  \includegraphics[width=.3\textwidth]{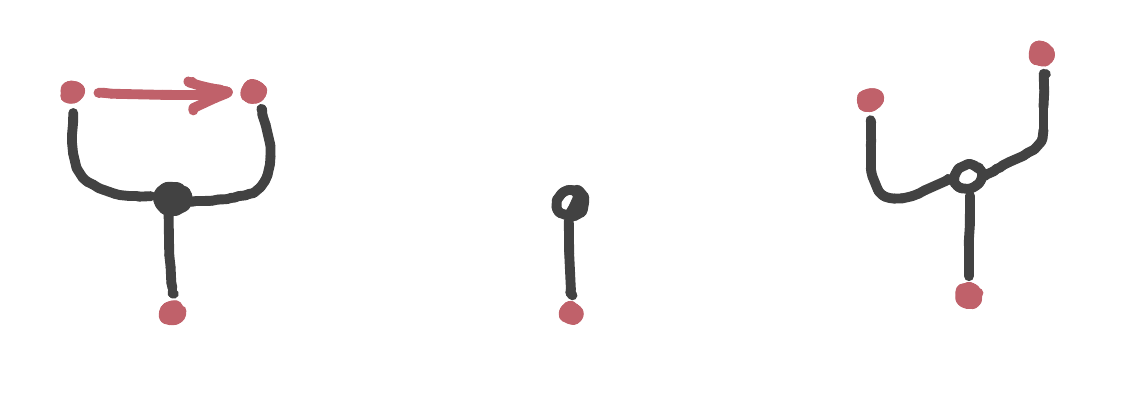} \\
  \includegraphics[width=.6\textwidth]{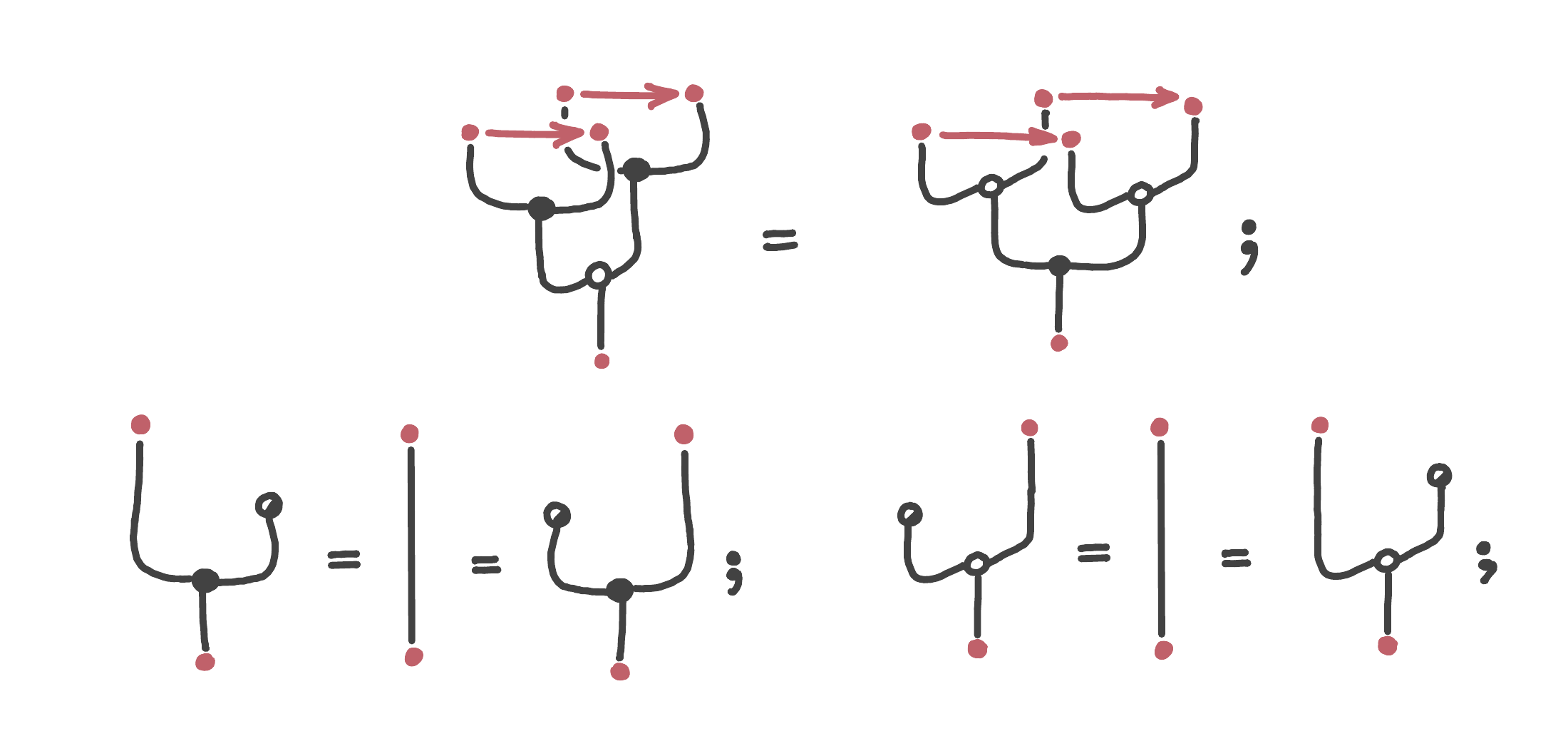}
  \caption{Generators and axioms for a duoid in a physical duoidal.}
  \label{diags:normalduoid}
\end{figure}
\begin{definition}[Normal duoid]
  A \emph{normal duoid} in a \physicalDuoidalCategory{} is a $⊲$-monoid in the category of $⊗$-monoids that additionally shares its unit with that of the base monoid.
  \begin{align*}
     m ፡ X ⊗ X → X, \qquad
      s ፡ X ⊲ X → X, \qquad
      u ፡ N → X.
  \end{align*}
  This means that the following diagrams must all commute. 
\[
\begin{tikzcd}
	{X \otimes X \otimes X} & {X \otimes X} \\
	{X \otimes X} & X
	\arrow["{m \otimes \mathrm{id}}", from=1-1, to=1-2]
	\arrow["{\mathrm{id} \otimes m}"', from=1-1, to=2-1]
	\arrow["m", from=1-2, to=2-2]
	\arrow["m"', from=2-1, to=2-2]
\end{tikzcd} \quad
\begin{tikzcd}
	X & {X \otimes X} \\
	{X \otimes X} & X
	\arrow["{\mathrm{id} \otimes u}", from=1-1, to=1-2]
	\arrow["{u \otimes \mathrm{id}}"', from=1-1, to=2-1]
	\arrow["{\mathrm{id}}"', no head, from=1-1, to=2-2]
	\arrow["m", from=1-2, to=2-2]
	\arrow["m"', from=2-1, to=2-2]
\end{tikzcd}\]
\[
\begin{tikzcd}
	{X ⊲ X ⊲ X} & {X ⊲ X} \\
	{X ⊲ X} & X
	\arrow["{m ⊲ \mathrm{id}}", from=1-1, to=1-2]
	\arrow["{\mathrm{id} ⊲ m}"', from=1-1, to=2-1]
	\arrow["m", from=1-2, to=2-2]
	\arrow["m"', from=2-1, to=2-2] 
\end{tikzcd} \quad
\begin{tikzcd}
	X & {X ⊲ X} \\
	{X ⊲ X} & X
	\arrow["{\mathrm{id} ⊲ u}", from=1-1, to=1-2]
	\arrow["{u ⊲ \mathrm{id}}"', from=1-1, to=2-1]
	\arrow["{\mathrm{id}}"', no head, from=1-1, to=2-2]
	\arrow["m", from=1-2, to=2-2]
	\arrow["m"', from=2-1, to=2-2]
\end{tikzcd}\]
\[\begin{tikzcd}[row sep=tiny]
	{(X \lhd X) \otimes (X \lhd X)} & {X \otimes X} \\
	&& X \\
	{(X \otimes X) \lhd (X \otimes X)} & {X \lhd X}
	\arrow["{s \otimes s}", from=1-1, to=1-2]
	\arrow["{\mathcal{d}}"', from=1-1, to=3-1]
	\arrow["m", from=1-2, to=2-3]
	\arrow["{m \lhd m}"', from=3-1, to=3-2]
	\arrow["s"', from=3-2, to=2-3]
\end{tikzcd}\] 
  Alternatively, a normal duoidal is given by the string diagrams of \Cref{diags:normalduoid}.
\end{definition}
\begin{remark}
  A normal duoid in the normal duoidal category of bistrong profunctors \cite{garner2016commutativity} (or \emph{Tambara modules}) is a monoidal promonad: an identity on objects strict monoidal functor.
\end{remark}

\subsection{Enriched multicategories}

Multicategories can be enriched over any symmetric monoidal category. However, a more subtle and general enrichment for multicategories is that over physical duoidal categories, due to Rajesh in recent work~\cite{rajesh2023}. 

\begin{figure}[ht]
  \centering
  \includegraphics[width=.8\textwidth]{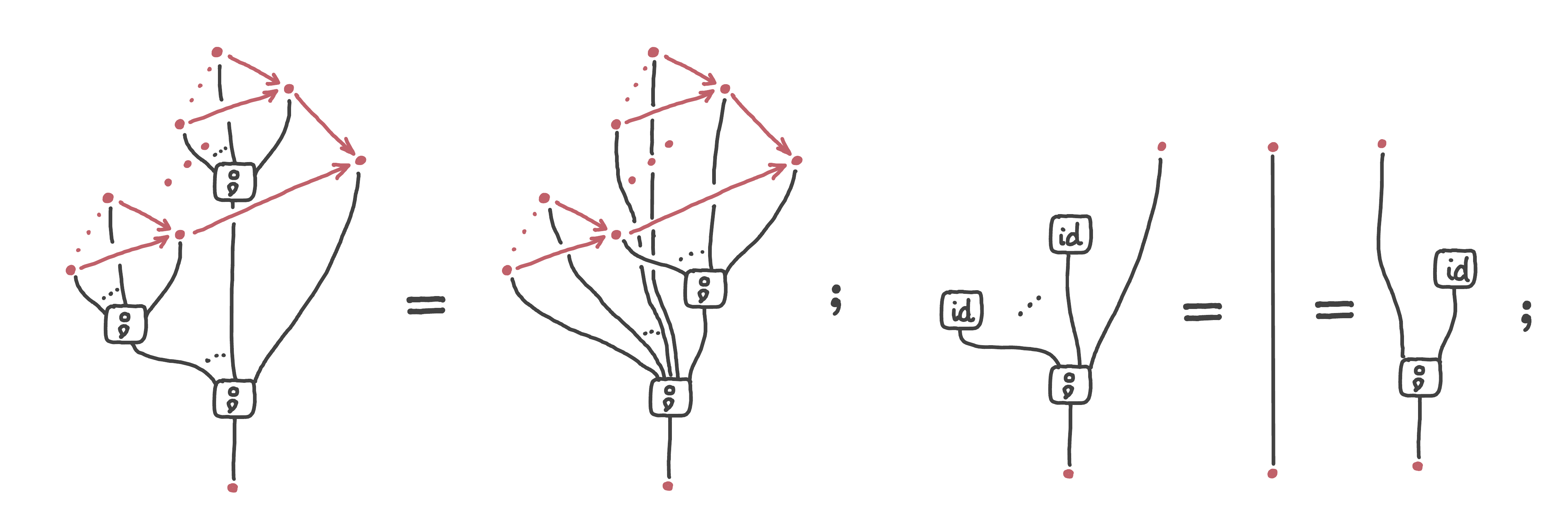}
  \caption{Axioms for duoidally enriched multicategories.}
  \label{fig:opduoidally-enriched-multicategory}
\end{figure}

\begin{definition}[Duoidally enriched multicategory \cite{rajesh2023}]
  \label{def:opduoidally-enriched}
  An \emph{enriched multicategory}, $ℂ$, over a \physicalDuoidalCategory{} consists of \emph{(i)} a set of objects, $ℂ_{obj}$; \emph{(ii)} a set of multimorphisms from each list of objects to each single object, $ℂ(X₁, ..., Xₙ; Y)$ for each $X₁,...,Xₙ,Y ∈ ℂ_{obj}$; \emph{(iii)} an identity operation, $\mathrm{i} ፡ N → ℂ(X;X)$ for each $X ∈ ℂ_{obj}$; and \emph{(iv)} a composition operation,
  $$(⨾) ፡ \Big( \bigotimes_{i = 1}^{k} ℂ(X_1^i,...,X^i_{n_i}; Y_i) \Big) ⊳ ℂ(Y_1,...,Y_m; Z) → ℂ(X_1^1,...,X^1_{n_1},...,X^k_1,...,X^k_{n_k}; Y).$$
  These must make all the following diagrams commute.
\[\begin{tikzcd}
 {ℂ(X₁,...,Xₙ;Y) ⊳ ℂ(Y;Y)} & {ℂ(X₁,...,Xₙ;Y)} \\
 {ℂ(X₁,...,Xₙ;Y)}
 \arrow["{(⨾)}", from=1-1, to=1-2]
 \arrow["{\mathrm{id} \rhd (\mathrm{i})}", from=2-1, to=1-1]
 \arrow["{\mathrm{id}}"', no head, from=2-1, to=1-2]
\end{tikzcd}\]
\[\begin{tikzcd} 
 {\big( \bigotimes_{i=1}^k ℂ(X_i;X_i) \big) ⊳ ℂ(X₁,...,Xₙ;Y)} & {ℂ(X₁,...,Xₙ;Y)} \\
 {ℂ(X₁,...,Xₙ;Y)}
 \arrow["{(⨾)}", from=1-1, to=1-2]
 \arrow["{\big(\bigotimes_{i=1}^k \mathrm{i}\big) ⊳ \mathrm{id}}", from=2-1, to=1-1]
 \arrow["{\mathrm{id}}"', no head, from=2-1, to=1-2] 
\end{tikzcd}\]
\[\begin{tikzcd}
 \begin{array}{c} \big( \bigotimes_{j=1}^{p} \big( \bigotimes_{i=1}^{m_j} ℂ(X_1^{i,j},...,X_{n_{i,j}}^{i,j};Y_i^{j}) \big) \\ \rhd ℂ(Y_1^{j},...,Y_{m_j}^{j}; Z_j) \big) \\ \rhd ℂ(Z_1,...,Z_p;U) \end{array} & \begin{array}{c} \big( \bigotimes_{j=1}^{p}  ℂ(X_1^{1,1},...,X_{n_{m_j,j}}^{m_j,j}; Z_j) \big) \\ \rhd ℂ(Z_1,...,Z_p;U) \end{array} \\
 \begin{array}{c} \big( \bigotimes_{j=1}^{p} \bigotimes_{i=1}^{m_p} ℂ(X_1^{i,j},...,X_{n_{i,j}}^{i,j};Y_i^{j}) \big) \\ \rhd \big( \bigotimes_{i=1}^{m_p} ℂ(Y_1^{j},...,Y_{m_p}^{j}; Z_j) \big) \\ \rhd ℂ(Z_1,...,Z_p;U) \end{array} \\
 \begin{array}{c} \big( \bigotimes_{j=1}^{p} \bigotimes_{i=1}^{m_p} ℂ(X_1^{i,j},...,X_{n_{i,j}}^{i,j};Y_i^{j}) \big) \\ \rhd ℂ(Y_1^{1},...,Y_{m_p}^{p};U) \end{array} & {ℂ(X_1^{1,1},...,X_{n_{m_p,p}}^{m_p,p};U)}
 \arrow["{(⨾) \rhd \mathrm{id}}", from=1-1, to=1-2]
 \arrow["{\mathcal{d}}"', from=1-1, to=2-1]
 \arrow["{(⨾)}", from=1-2, to=3-2]
 \arrow["{\mathrm{id} \rhd (⨾)}"', from=2-1, to=3-1]
 \arrow["{(⨾)}"', from=3-1, to=3-2]
\end{tikzcd}\]     
  Alternatively, a duoidally-enriched multicategory is given by the string diagrams of \Cref{fig:opduoidally-enriched-multicategory}.
\end{definition}

\begin{figure}[ht]
  \centering
  \includegraphics[width=.8\textwidth]{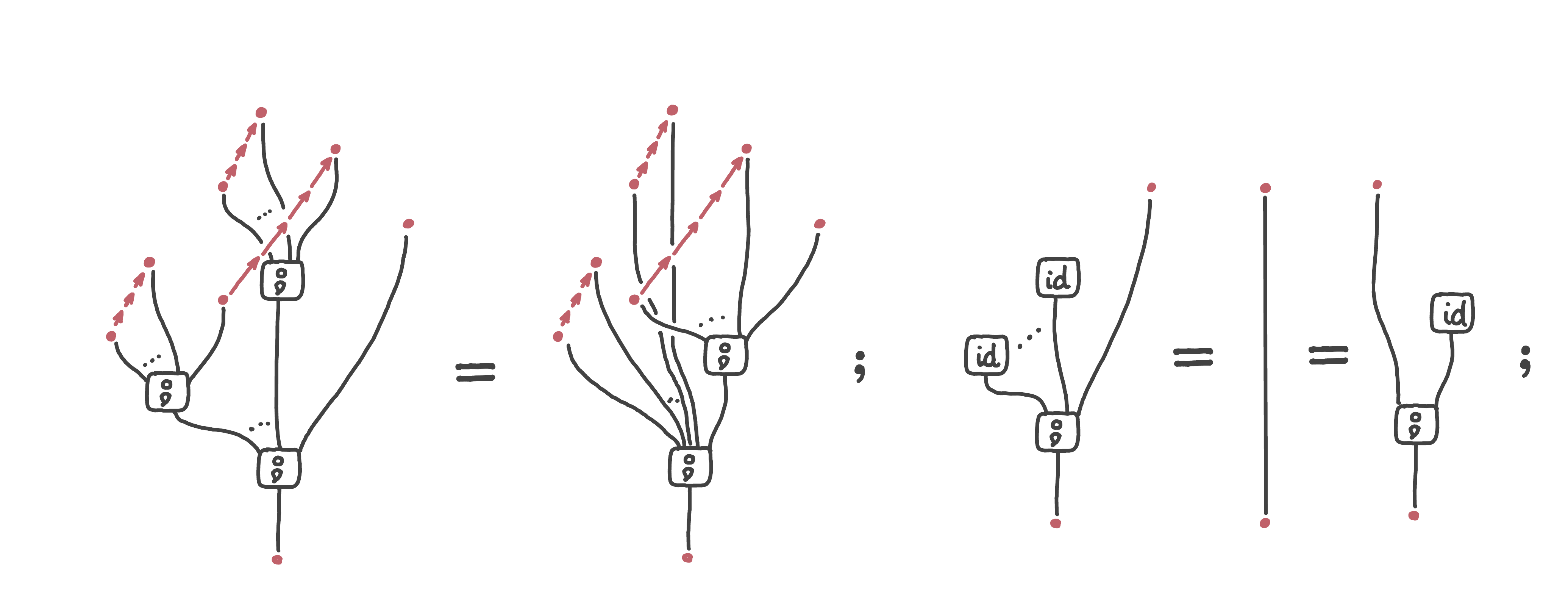}
  \caption{Axioms for opduoidally enriched multicategories.}
  \label{fig:duoidally-enriched-multicategory}
\end{figure}

\begin{definition}[Opduoidally enriched multicategory, c.f. \cite{rajesh2023}]
   \label{def:duoidally-enriched}
   An \emph{opduoidally enriched multicategory}, $ℂ$, over a \physicalDuoidalCategory{} consists of \emph{(i)} a set of objects, $ℂ_{obj}$; \emph{(ii)} a set of multimorphisms from each list of objects to each single object, $ℂ(X₁, ..., Xₙ; Y)$ for each $X₁,...,Xₙ,Y ∈ ℂ_{obj}$; \emph{(iii)} an identity operation, $\mathrm{i} ፡ N → ℂ(X;X)$ for each $X ∈ ℂ_{obj}$; and \emph{(iv)} a composition operation,
   $$(⨾) ፡ \Big( {⊲}_{i = 1}^{k} ℂ(X_1^i,...,X^i_{n_i}; Y_i) \Big) ⊗ ℂ(Y_1,...,Y_m; Z) → ℂ(X_1^1,...,X^1_{n_1},...,X^k_1,...,X^k_{n_k}; Y).$$
   These must make all the following diagrams commute.
\[\begin{tikzcd}
	{ℂ(X₁,...,Xₙ;Y) ⊗ ℂ(Y;Y)} & {ℂ(X₁,...,Xₙ;Y)} \\
	{ℂ(X₁,...,Xₙ;Y)}
	\arrow["{(⨾)}", from=1-1, to=1-2]
	\arrow["{\mathrm{id} ⊗ (\mathrm{i})}", from=2-1, to=1-1]
	\arrow["{\mathrm{id}}"', no head, from=2-1, to=1-2]
\end{tikzcd}\]
\[\begin{tikzcd} 
	{\big( {⊲}_{i=0}^n ℂ(X_i;X_i) \big) ⊗ ℂ(X₁,...,Xₙ;Y)} & {ℂ(X₁,...,Xₙ;Y)} \\
	{ℂ(X₁,...,Xₙ;Y)}
	\arrow["{(⨾)}", from=1-1, to=1-2]
	\arrow["{\big({⊳}_{i=0}^n \mathrm{i}\big) ⊗ \mathrm{id}}", from=2-1, to=1-1] 
	\arrow["{\mathrm{id}}"', no head, from=2-1, to=1-2] 
\end{tikzcd}\]
\[\begin{tikzcd}
	\begin{array}{c} 
    \big( ⊗_{i=0}^{o} \big( ⊲_{j=0}^{m_i} ℂ(X_1^{i,j},...,X_{n_{i,j}}^{i,j};Y_j^{i}) \big) \big) \\ 
    ⊗ (⊲_{i=0}^o ℂ(Y_1^{i},...,Y_{m_i}^{i}; Z_i) ) \\ 
    ⊗\ ℂ(Z_1,...,Z_o;U) 
  \end{array} & 
  \begin{array}{c} 
    \big( ⊗_{i=0}^{o} \big( ⊲_{j=0}^{m_i} ℂ(X_1^{i,j},...,X_{n_{i,j}}^{i,j};Y_j^{i}) \big) \big) \\ 
    ⊗\ ℂ(Y_1^{1},...,Y_{m_i}^{1},...,Y_1^{o},...,Y_{m_o}^{o};U) 
  \end{array} 
  \\
	\begin{array}{c}
    \big( ⊲_{i=0}^{o} \big( 
      ⊲_{j=0}^{m_i}
        ℂ(X_1^{i,j},...,X_{n_{i,j}}^{i,j};Y_j^{i})
    \big) \\
    ⊗\ ℂ(Y_1^{i},...,Y_{m_i}^{i}; Z_i) \big) \\
    ⊗\ ℂ(Z_1,...,Z_o;U)
  \end{array} & 
  \begin{array}{c} 
    \big( ⊲_{i=0}^{o} \big( ⊲_{j=0}^{m_i} ℂ(X_1^{i,j},...,X_{n_{i,j}}^{i,j};Y_j^{i}) \big) \big) \\ 
    ⊗\ ℂ(Y_1^{1},...,Y_{m_i}^{1},...,Y_1^{o},...,Y_{m_o}^{o};U) 
  \end{array} 
  \\
	\begin{array}{c}
    \big( ⊲_{i=0}^{o}
        ℂ(X_1^{i,1},...,X_{n_{i,m_i}}^{i,m_i};Z_i)
    \big) \\
    ⊗\ ℂ(Z_1,...,Z_o;U)
  \end{array} 
  & {ℂ(X_1^{1,1},...,X_{n_{o,m_o}}^{o,m_o};U)}
	\arrow["{\mathrm{id} ⊗ (⨾)}", from=1-1, to=1-2]
	\arrow["{\mathcal{d}}"', from=1-1, to=2-1]
	\arrow["{𝒹 ⊗ \mathrm{id}}", from=1-2, to=2-2]
  \arrow["{(⨾)}", from=2-2, to=3-2]
	\arrow["{\mathrm{id} ⊗ (⨾)}"', from=2-1, to=3-1]
	\arrow["{(⨾)}"', from=3-1, to=3-2]
\end{tikzcd}\]    
   Alternatively, an opduoidally-enriched multicategory is given by the string diagrams of \Cref{fig:duoidally-enriched-multicategory} (c.f.~\Cref{def:opduoidally-enriched}). 
\end{definition}

\subsection{Bimonoids, Frobenius monoids and Dualities}
We can also define most of the concepts we can define in braided monoidal categories, with the twist of using two different tensors: note that this is not less nor more general than defining these concepts in braided monoidal categories; we assume symmetry of one of the tensors, but we gain a new tensor that does not need to be braided.

\begin{definition}[Bimonoid in a physical duoidal category]
  A \emph{bimonoid} in a \physicalDuoidalCategory{} is a $⊲$-comonoid in the category of $⊗$-monoids.
  \begin{align*}
    m ፡ X ⊗ X → X, \qquad
    u ፡ N → X, \qquad
    d ፡ X → X ⊲ X, \qquad
    e ፡ X → N.
  \end{align*}
  This means that the following diagrams must all commute.
  \[ 
\begin{tikzcd}
	{X \otimes X \otimes X} & {X \otimes X} \\
	{X \otimes X} & X
	\arrow["{m \otimes \mathrm{id}}", from=1-1, to=1-2]
	\arrow["{\mathrm{id} \otimes m}"', from=1-1, to=2-1]
	\arrow["m", from=1-2, to=2-2]
	\arrow["m"', from=2-1, to=2-2]
\end{tikzcd} \quad
\begin{tikzcd}
	X & {X \otimes X} \\
	{X \otimes X} & X
	\arrow["{\mathrm{id} \otimes u}", from=1-1, to=1-2]
	\arrow["{u \otimes \mathrm{id}}"', from=1-1, to=2-1]
	\arrow["{\mathrm{id}}"', no head, from=1-1, to=2-2]
	\arrow["m", from=1-2, to=2-2]
	\arrow["m"', from=2-1, to=2-2]
\end{tikzcd}\]
\[%
\begin{tikzcd}  
	X & {X \lhd X} \\
	{X \lhd X} & {X \lhd X \lhd X}
	\arrow["d", from=1-1, to=1-2]
	\arrow["d"', from=1-1, to=2-1]
	\arrow["{d \lhd \mathrm{id}}", from=1-2, to=2-2]
	\arrow["{\mathrm{id} \lhd d}"', from=2-1, to=2-2]
\end{tikzcd}
\begin{tikzcd} 
	X & {X \lhd X} \\
	{X \lhd X} & X
	\arrow["d", from=1-1, to=1-2]
	\arrow["d"', from=1-1, to=2-1]
	\arrow["{\mathrm{id} \lhd e}", from=1-2, to=2-2]
	\arrow["{e \lhd \mathrm{id}}"', from=2-1, to=2-2]
\end{tikzcd}\]  
\[\begin{tikzcd}
	{X \otimes X} & X & {X \lhd X} \\
	{(X \lhd X) \otimes (X \lhd X)} & {(X \otimes X) \lhd (X \otimes X)} & {X \lhd X}
	\arrow["m", from=1-1, to=1-2]
	\arrow["{d \otimes d}"', from=1-1, to=2-1]
	\arrow["d", from=1-2, to=1-3]
	\arrow["{\mathrm{id}}", no head, from=1-3, to=2-3]
	\arrow["{\mathcal{d}}"', from=2-1, to=2-2]
	\arrow["{m \lhd m}"', from=2-2, to=2-3]
\end{tikzcd}\]
\[\begin{tikzcd} 
	N & {X \lhd X} \\
	X & {X \lhd X}
	\arrow["{u \lhd u}", from=1-1, to=1-2]
	\arrow["u"', from=1-1, to=2-1]
	\arrow["{\mathrm{id}}", no head, from=1-2, to=2-2]
	\arrow["d"', from=2-1, to=2-2]
\end{tikzcd}  
\begin{tikzcd}
	{X \otimes X} & N \\
	X & N
	\arrow["{e \otimes e}", from=1-1, to=1-2]
	\arrow["m"', from=1-1, to=2-1]
	\arrow["{\mathrm{id}}", no head, from=1-2, to=2-2]
	\arrow["e", from=2-1, to=2-2]
\end{tikzcd}\]
Alternatively, a bimonoid is given by the string diagrams of \Cref{diags:bimonoid}.
\end{definition}
\begin{figure}[ht]
  \centering
  \includegraphics[width=.6\textwidth]{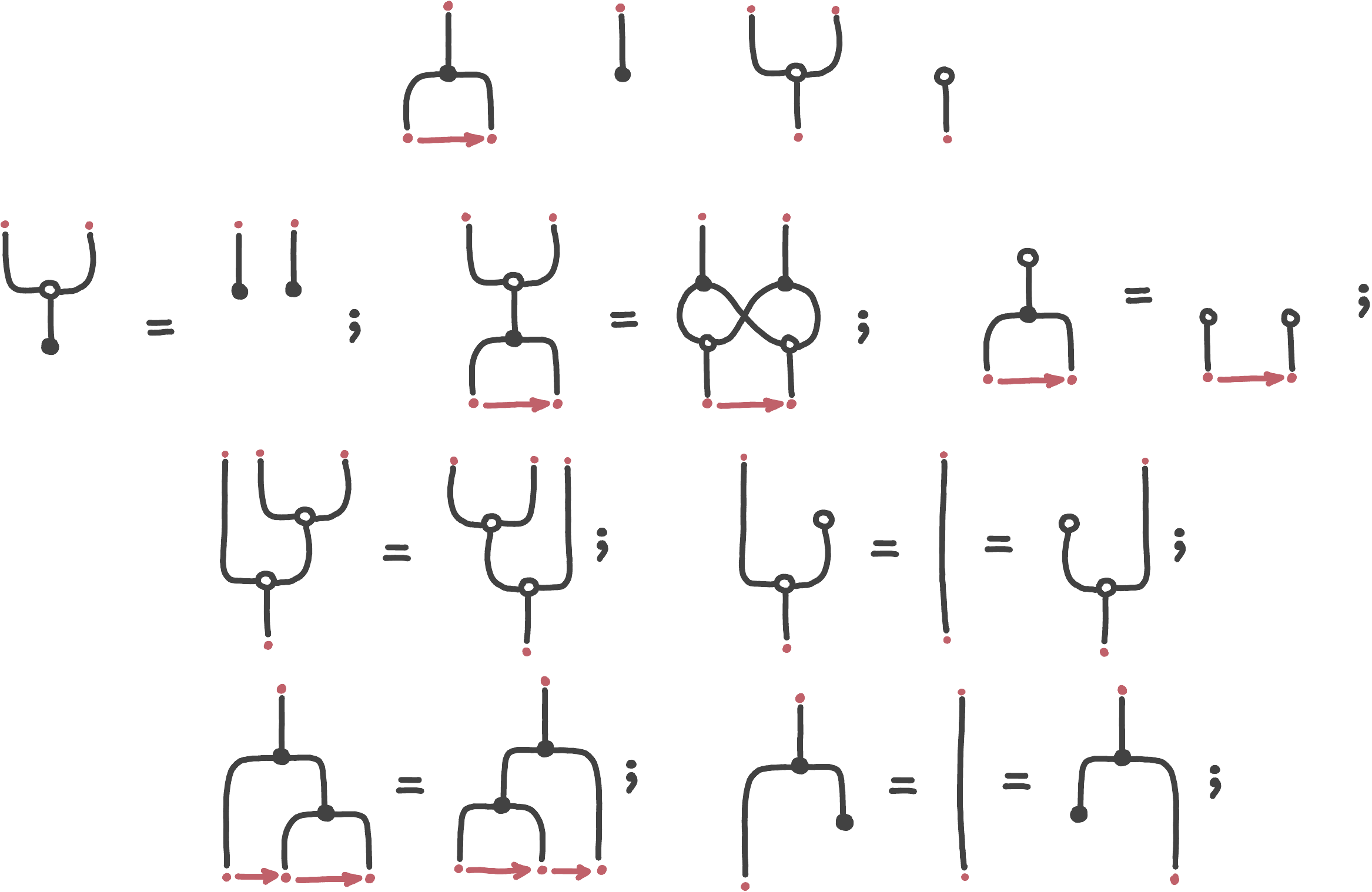}
  \caption{Generators and axioms for a bimonoid in a physical duoidal category.}
  \label{diags:bimonoid}
\end{figure}

\begin{definition}[Frobenius monoid in a physical duoidal category]
  A \emph{Frobenius monoid} in a \physicalDuoidalCategory{} is both a $⊲$-comonoid and a $⊗$-monoid structure over the same object, interacting by the Frobenius axiom.
  \begin{figure}[ht]
    \centering 
    \includegraphics[width=.6\textwidth]{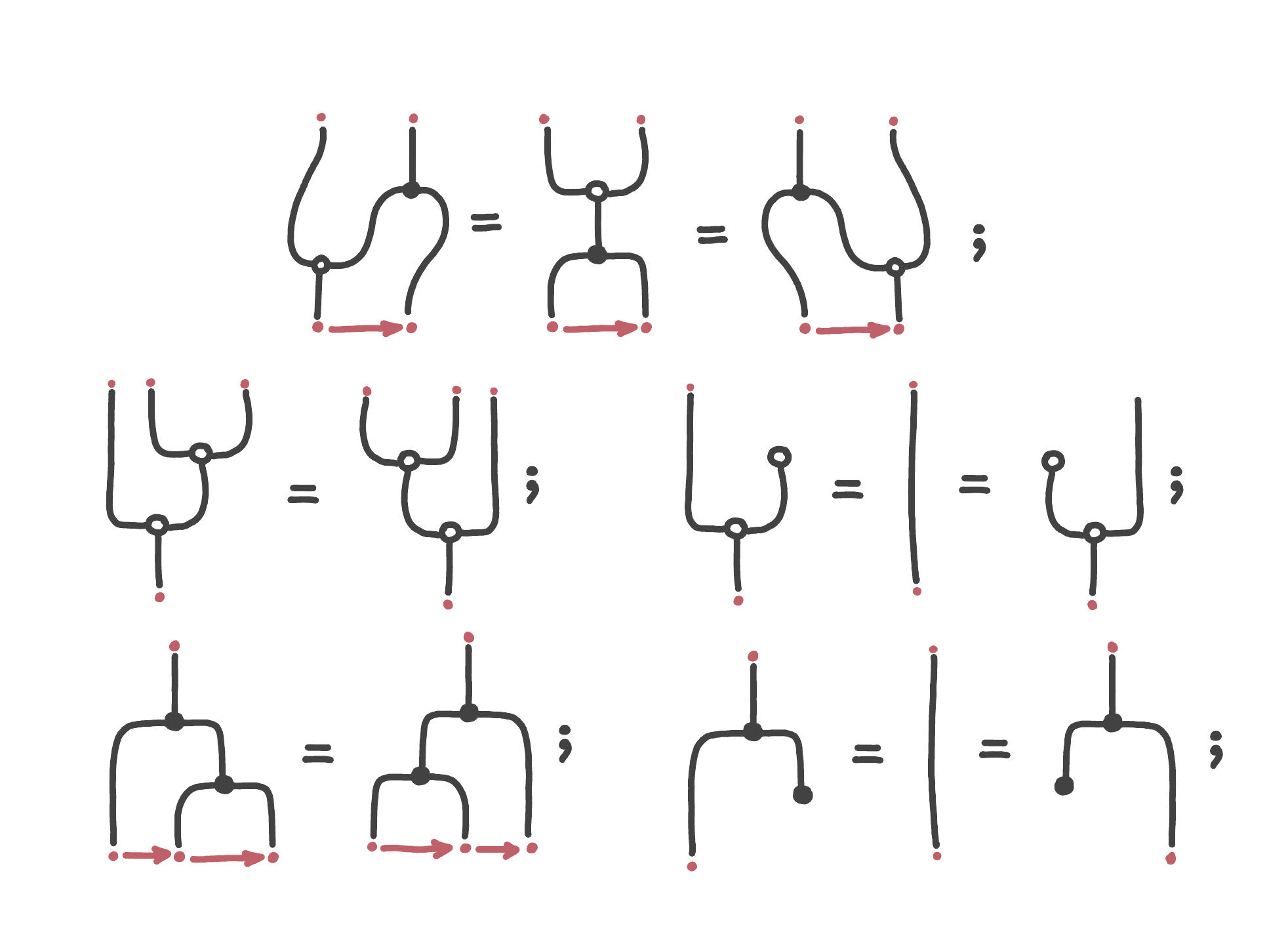}
    \caption{Generators and axioms for a Frobenius monoid in a physical duoidal category.}
    \label{diags:frobenius}
  \end{figure}
\end{definition}

\begin{definition}[Duoidal duality]
  A \emph{duoidal duality}, $A ⊣ A^{\ast}$, is a pair of morphisms, $η ፡ N → A ⊲ A^{\ast}$ and $ε ፡ A^{\ast} ⊗ A → N$, satisfying the snake equations up to the duoidal distributor.
\end{definition}

\begin{proposition}
  A duoidal duality induces a duoidal Frobenius monoid structure on $A ⊲ A^{\ast}$ (see \Cref{fig:duoidalDuality}).
\end{proposition}

\begin{figure}[!ht]
  \centering
  \includegraphics[width=.5\textwidth]{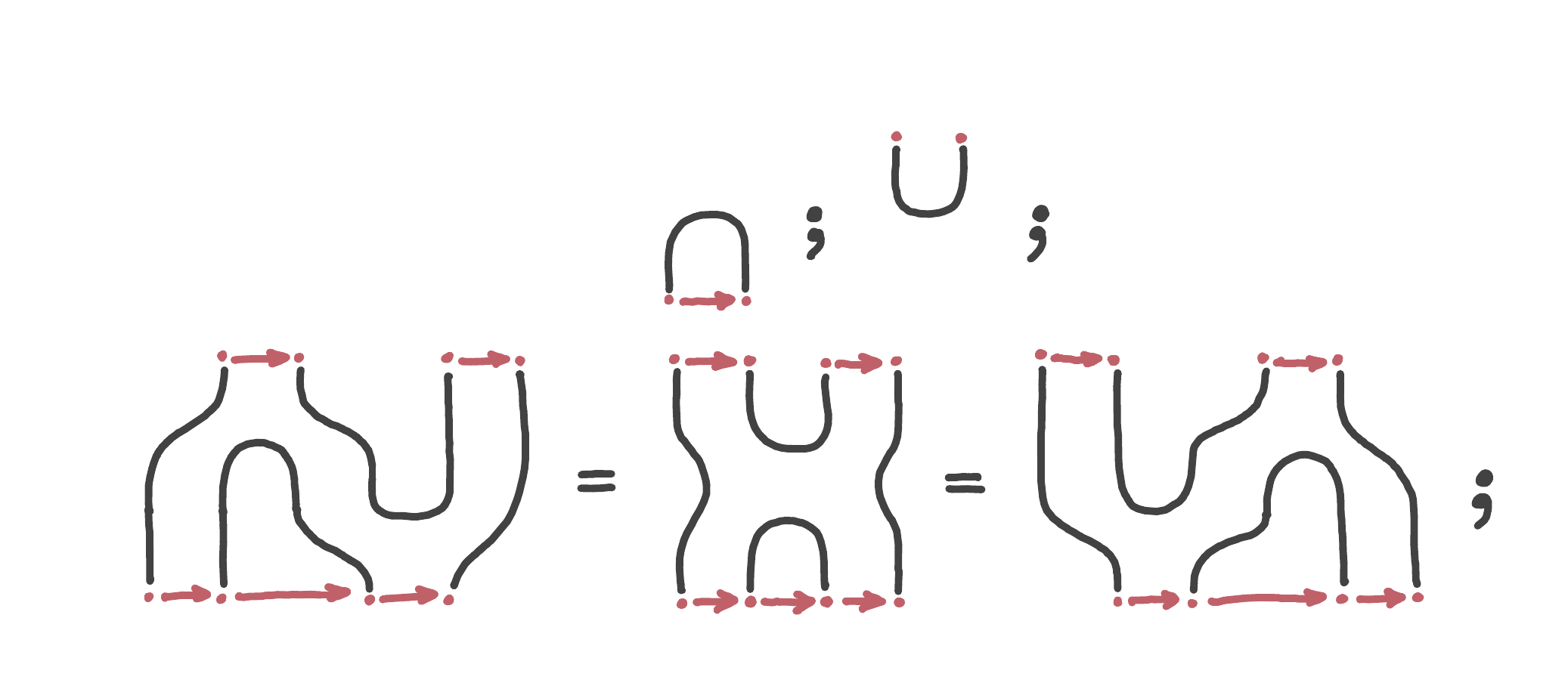}
  \caption{A duoidal duality induces a Frobenius monoid.}
  \label{fig:duoidalDuality}
\end{figure}

\begin{definition}
  Traditionally, we can only define commutative monoids in symmetric (or braided) monoidal categories. A $⊗$-commutative $⊲$-monoid in a \physicalDuoidalCategory{} is such that the following equation holds.
\end{definition}

\subsection{Physical Duoidal Categories are Spacial}

\begin{proposition}
  Every \physicalDuoidalCategory{} is $(⊲)$-spacial, in the sense of Selinger \cite{selinger2010survey}. Given any object $A$ and any scalar $α ፡ I → I$, we have that $$\mathrm{id}_A ⊲ α = α ⊲ \mathrm{id}_A.$$
\end{proposition}
\begin{proof}
  The following is a commutative diagram. We use naturality of the symmetry ($\sigma$) and of the distributors ($𝒹$), and we use coherence for normal duoidal categories.
\[\begin{tikzcd}[row sep=tiny]
	&& {N \lhd A} & {N \lhd A} \\
	\\
	& {I \otimes A} & {N \otimes A} & {N \otimes A} & {I \otimes A} \\
	A &&&&& A \\
	& { A \otimes I} & {A \otimes N} & {A \otimes N} & {A \otimes I} \\
	\\
	&& {A \lhd N} & {A \lhd N}
	\arrow["{\alpha \lhd \mathrm{id}}", from=1-3, to=1-4]
	\arrow["{\lambda^{-1}}", curve={height=-18pt}, from=1-4, to=4-6]
	\arrow["{\psi \otimes \mathrm{id}}", from=3-2, to=3-3]
	\arrow["\sigma"', from=3-2, to=5-2]
	\arrow["{\mathcal{d}}"', from=3-3, to=1-3]
	\arrow["{\alpha \otimes \mathrm{id}}", from=3-3, to=3-4]
	\arrow["\sigma"', from=3-3, to=5-3]
	\arrow["{\mathcal{d}}"', from=3-4, to=1-4]
	\arrow["{\psi \otimes \mathrm{id}}", from=3-4, to=3-5]
	\arrow["\sigma", from=3-4, to=5-4]
	\arrow["\lambda", from=3-5, to=4-6]
	\arrow["\sigma", from=3-5, to=5-5]
	\arrow["\lambda", curve={height=-18pt}, from=4-1, to=1-3]
	\arrow["\lambda", from=4-1, to=3-2]
	\arrow["\rho"', from=4-1, to=5-2]
	\arrow["\rho"', curve={height=18pt}, from=4-1, to=7-3]
	\arrow["{\mathrm{id} \otimes \psi}"', from=5-2, to=5-3]
	\arrow["{\mathrm{id} \otimes \alpha}"', from=5-3, to=5-4]
	\arrow["{\mathcal{d}}"', from=5-3, to=7-3]
	\arrow["{\mathrm{id} \otimes \psi}"', from=5-4, to=5-5]
	\arrow["{\mathcal{d}}"', from=5-4, to=7-4]
	\arrow["\rho"', from=5-5, to=4-6]
	\arrow["{\mathrm{id} \lhd \alpha}"', from=7-3, to=7-4]
	\arrow["{\rho^{-1}}"', curve={height=18pt}, from=7-4, to=4-6]
\end{tikzcd}\]
Alternatively, in terms of string diagrams, the equation in \Cref{diags:spacial} holds.
\end{proof}
\begin{figure}[ht]
  \centering 
  \includegraphics[width=.3\textwidth]{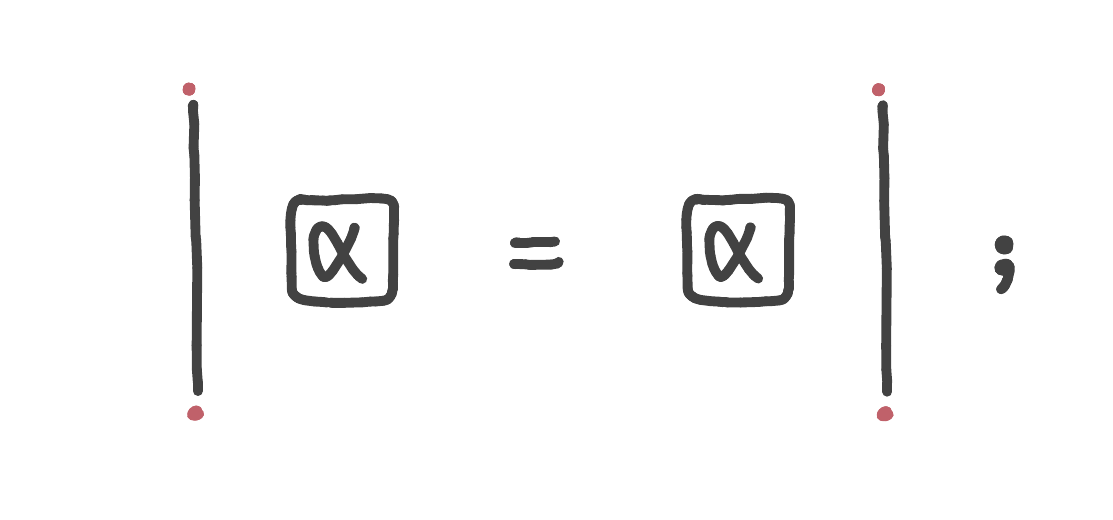}
  \caption{Spacial equation.}
  \label{diags:spacial}
\end{figure}

\section{Opphysical Duoidal Categories}

While the opposite of a normal duoidal category, $(𝕍,⊗,⊲,N)$, is again a normal duoidal category, $(𝕍^{op},⊲,⊗,N)$, it is not true that the opposite of a \physicalDuoidalCategory{} is again a \physicalDuoidalCategory{}: there is no reason to expect that both tensors will be symmetric. Instead, the opposite of a \physicalDuoidalCategory{} is an \emph{opphysical duoidal category}.

\begin{definition}
  A \emph{strict ophysical duoidal category} is a category with a strict monoidal structure and a strict symmetric monoidal structure sharing the same unit, $(𝕍,⊠,⊲,N)$, and such that the first monoidal structure distributes over the second; that is, there exist maps
  \begin{align*}
    & 𝒹_{X,Y,Z,W} ፡ (X ⊠ Z) ⊲ (Y ⊠ W) → (X ⊲ Y) ⊠ (Z ⊲ W); \\
    & 𝓈_{X,Y} ፡ X ⊠ Y → Y ⊠ X.
  \end{align*}
  Opphysical duoidal categories are defined to be coherent structures, meaning that any formally distinctly typed equation of morphisms on the free strict opphysical duoidal category holds true.
\end{definition}

\begin{corollary}
  String diagrams for opphysical duoidal categories are exactly the bottom-top string diagrams for physical duoidal categories.
\end{corollary}

\section{Conclusions}
\subsection{Related Work}
\defining{linkProfunctor}{} 
Our main reference text is the monograph on duoidal categories by Aguiar and Mahajan \cite{aguiar2010monoidal}, where various results on coherence for duoidal categories and normal duoidal categories are discussed.
Our presentation of \physicalDuoidalCategories{} follows that of Shapiro and Spivak \cite{shapiro2022duoidal}, who also introduced the name \emph{physical} for normal and $⊗$-symmetric duoidal categories. The study of physical duoidal expressions and their correspondence to zetless posets is due to Grabowski and Gischer.

On the applied side, we follow the idea of Garner and López Franco \cite{garner2016commutativity} of using normal duoidal categories to study commutativity on algebraic structures; but also previous work by Earnshaw, Hefford, and this author on the interpretation of causality on monoidal string diagrams using duoidal structures \cite{earnshaw2024produoidal}.

\subsection{Further work}
We can easily conjecture that the string diagrams for \emph{dependence categories} \cite{shapiro2022duoidal} are similar to those presented here, with the only difference of removing the restriction to \zetlessPosets{}. Although much more natural from this point of view, dependence categories are less frequent than \physicalDuoidalCategories{}, and we leave their study to further work.

String diagrams for \physicalDuoidalCategories{} particularize into the hypergraph string diagrams for \symmetricMonoidalCategories{}: in fact, any symmetric monoidal category is automatically duoidal with itself. A question remains on whether string diagrams for \physicalDuoidalCategories{} also particularize to certain planar monoidal categories: it can be shown that a monoidal category can be part of a duoidal structure only if it is spacial in the sense of Selinger \cite{selinger2010survey}.

After our characterization, it becomes obvious that a do-notation where a poset of variables is automatically tracked by the type-checker could constitute a good programming-like internal language for physical duoidal categories. We leave this development for further work.

\subsection{Acknowledgements}
The author thanks Nayan Rajesh for spotting an error in a previous definition of opduoidally enriched multicategory. The author wants to thank Nayan Rajesh, Matt Earnshaw, Philip Saville, and Sam Staton for helpful discussion on physical duoidal categories.

\bibliographystyle{alpha}
\bibliography{main.bib}

\newpage{}
\appendix
\section{Omitted Proofs}

\begin{proposition}[From {{\Cref{prop:PhyDuoCategory}}}]
    \label{ax:prop:PhyDuoCategory}
    \defining{linkPhyDuoCategory}{}
    \StrictPhysicalDuoidalCategories{} and \strictPhysicalDuoidalFunctors{} between them form a category, $\PhyDuo$.
\end{proposition}
\begin{proof}
    Let us show that any identity functor is a \strictPhysicalDuoidalFunctor{}: it is strict monoidal for the two structures, respectively, and it also strictly preserves structure maps, $\mathrm{Id}(𝒹_V) = 𝒹_V$ and $\mathrm{Id}(σ_V) = σ_V$.

    Let us show that any composition of two \strictPhysicalDuoidalFunctors{}, 
    \begin{align*}
      & F ፡ (𝕍,⊗_V,⊲_V,𝒹_V,σ_V) → (𝕎,⊗_W,⊲_W,𝒹_W,σ_W),\mbox{ and } \\
      & G ፡ (𝕎,⊗_W,⊲_W,𝒹_W,σ_W) → (\mathbb{X},⊗_X,⊲_X,𝒹_X,σ_X)
    \end{align*} 
    is again a \strictPhysicalDuoidalFunctor{}. In fact, the composition of strict symmetric monoidal functors is again a strict symmetric monoidal functor; the composition of strict monoidal functors is again a strict monoidal functor; and we can see that $G(F(𝒹_V)) = G(𝒹_W) = 𝒹_X$ and $G(F(σ_V)) = G(σ_W) = σ_X$.
  \end{proof}
  
  \begin{proposition}[From {{\Cref{prop:connectedBySpan}}}]
    \label{ax:prop:connectedBySpan}
    In a \zetlessPoset{}, any two connected elements must be connected by either a span or a cospan. That is, if there is a path between two elements, $x₀ → x₁ ← x₂ → ... ← xₙ$, there must exist either a cospan between them, $x₀ → u ← xₙ$, or a span between them, $x₀ ← v → xₙ$.
  \end{proposition}
  \begin{proof}
    Assume $x$ is connected to $y$. If the shortest path connecting them has less than three steps, we are done. Otherwise, the shortest path connecting them must start, without loss of generality, by $x₀ → x₁ ← x₂ → x₃$.
  
    We know that $x₀ ≰ x₂$, because $x₀ ≤ x₂$ would break the minimality of the path. We also know that $x₂ ≰ x₀$, because $x₂ ≤ x₀$ would build a cospan, breaking the minimality again. By an analogous reasoning, $x₁ ≰ x₃$ and $x₃ ≰ x₁$. Finally, we know that $x₀ ≰ x₃$ and $x₃ ≰ x₀$, again because of minimality of the path.
  
    However, all this means that we have a fully faithfull embedding of the $\mathsf{Z}$-poset: we have reached a contradiction.
  \end{proof}

  \begin{proposition}[From {{\Cref{prop:zetlessPosetsDuoidalExpressions}}}]
    \label{ax:prop:zetlessPosetsDuoidalExpressions}
    \ZetlessPosets{} labelled over a set are in correspondence with \duoidalExpressions{} on that set, up to symmetries of the tensored components,
    \[
      \zetless(A) ≅ \Expr(A) / (≈).
    \]
  \end{proposition}
  \begin{proof}
    Any zetless poset is exactly in one of these four cases: (1) it is empty; (2) it is a singleton; (3) it is $⊲$-composite, connected and thus \seqPrime{}; or (4) it is $⊗$-composite, disconnected and thus \parPrime{}. Note that, if any \zetlessPoset{} was to be both \parPrime{} and \seqPrime{}, it must be a singleton (\Cref{prop:prime-prime-singleton}).
    
    As a consequence, every \zetlessPoset{} can be built, exclusively, either as (1) the empty expression, $N$; (2) a singleton, $X$; (3) a sequencing of \zetlessPosets{}, $P₁ ⊲ ... ⊲ Pₙ$; or (4) a tensoring of \zetlessPosets{}, $P₁ ⊗ ... ⊗ Pₙ$: in this last case, because the tensoring operation is commutative, it can be built not uniquely, but up to a permutation, which is accounted for by the quotienting $(≈)$.
  \end{proof}
  
  \begin{proposition}[From {{\Cref{prop:inclusionmaps}}}]
    \label{ax:prop:inclusionmaps}
    The existence of an inclusion of \zetlessPosets{} corresponds to the existence of a structure map between their corresponding \duoidalExpressions{} in a \physicalDuoidalCategory{}.
  \end{proposition}
  \begin{proof}[Proof. Adapted from Shapiro and Spivak \cite{shapiro2022duoidal}]
    Any structure map in the category of \zetlessPosets{} induces an inclusion.
    Let us prove that the existence of an inclusion of \zetlessPosets{} implies the existence of a structure map between their corresponding expressions. We will employ induction on the size of the \zetlessPosets{} forming the inclusion, $P ⭇ Q$.
  
    The posets $P$ and $Q$ must be in any of the previous four cases. Note that, if $P$ is a sequence of \posets{}, then it must be connected and $Q$ must also be connected and factor as a sequence of \posets{}. Note that, if $Q$ is a tensoring of \posets{}, then it must be \incomparableConnected{} and $P$ must also be \incomparableConnected{} and factor as a tensoring of \posets{}.
    \begin{enumerate}
      \item Both are empty; $\id_N ፡ N → N$ is a structure map.
      \item Both are a singleton; $\id_A ፡ N → N$ is a structure map.
      \item The first one is a sequencing of \posets{}, $P = P₁ ⊲ ... ⊲ Pₙ$, thus we can write the second one (which has more edges) as a sequencing of the \posets{} containing the same objects, $Q = Q₁ ⊲ ... ⊲ Qₙ$, where the objects of $Pᵢ$ are the objects of $Qᵢ$. Any inclusion of \posets{} between the sequencing of \posets{} must factor as a sequencing of inclusions. We proceed by induction.
      \item The second one is a tensoring of \posets{}, $Q = Q₁ ⊗ ... ⊗ Qₙ$, thus we can write the first one (which has less edges) as a tensoring of of the \posets{} containing the same objects, $P = P₁ ⊗ ... ⊗ Pₙ$, where the objects of $Pᵢ$ are the objects of $Qᵢ$. Any inclusion of \posets{} between the tensoring of \posets{} must factor as a tensoring of inclusions. We proceed by induction.
      \item Finally, assume that the first one is a tensoring of \posets{}, $P = P₁ ⊗ ... ⊗ Pₙ$, and that the second one is a sequencing of \posets{}, $Q = Q₁ ⊲ ... ⊲ Qₘ$. Let us define $P_{ij}$ to be the full subposet of $P$ whose objects are the intersection of the objects in $Pᵢ$ and $Qⱼ$; let us define $Q_{ij}$ to be the full subposet of $Q$ whose objects are the intersection of the objects in $Pᵢ$ and $Qⱼ$.
      
      If there is an inclusion, then the following is a structure map.
      \begin{align*}
        P₁ ⊗ ... ⊗ Pₙ \ 
        &\longrightarrow\ (P_{11} ⊲ ... ⊲ P_{1m}) ⊗ ... ⊗ (P_{n1} ⊲ ... ⊲ P_{nm}), \\
        &\longrightarrow\ (P_{11} ⊗ ... ⊗ P_{n1}) ⊲ ... ⊲ (P_{1m} ⊲ ... ⊲ P_{nm}), \\
        &\longrightarrow\ Q₁ ⊲ ... ⊲ Qₘ.
      \end{align*}
    \end{enumerate}
    These cover all possible cases: any inclusion of \zetlessPosets{} corresponds to an structure map.
  \end{proof}

\begin{theorem}[From {{\Cref{thm:freePhysicalObjects}}}]
  \label{ax:thm:freePhysicalObjects}
  \defining{linkDuoidalObj}{}
  \ZetlessPosets{} construct the free \physicalDuoidalCategory{} over a set of objects. In other words, the functor $\Zetless{} ፡ \mathbf{Set} → \PhyDuo{}$ is left adjoint to the forgetful functor that picks the objects of a \physicalDuoidalCategory{}, $\mathsf{Obj} ፡ \PhyDuo{} → \mathbf{Set}$.
\end{theorem}
\begin{proof}
    Let us first construct the unit, $u_A ፡ A → \Obj(\Zetless(A))$: every element appears as a singleton. We will show that this is a universal arrow.
    Let $𝕍$ be a \physicalDuoidalCategory{} and let $f ፡ A → \Obj(𝕍)$ pick some objects of the category. We will prove that there exists a unique \strictPhysicalDuoidalFunctor{}, $f^{\ast} ፡ \Zetless(A) → 𝕍$, factoring $f = u ⨾ \Obj(f^{\ast})$. 
  
    The factoring forces $f^{\ast}(a) = f(a)$ for any singleton poset for an element $a ∈ A$. This is enough to force the inductive definition of the functor on objects.
    \begin{enumerate}
      \item $f^{\ast}(N) = N$, because of duoidality;
      \item $f^{\ast}(a) = f(a)$, because of universality;
      \item $f^{\ast}(E₁ ⊲ ... ⊲ Eₙ) = f^{\ast}(E₁) ⊲ ... ⊲ f^{\ast}(Eₙ)$, because of duoidality; and
      \item $f^{\ast}(E₁ ⊗ ... ⊗ Eₙ) = f^{\ast}(E₁) ⊗ ... ⊗ f^{\ast}(Eₙ)$, because of duoidality.
    \end{enumerate}
  
  Any \zetlessPoset{} can be writen uniquely as a duoidal expression on the singleton posets, uniquely up to symmetry (\Cref{prop:zetlessPosetsDuoidalExpressions}); because the \strictPhysicalDuoidalFunctor{} must preserve these expressions, it is determined on objects. Finally, an bijective-on-objects inclusion between \posets{} correspond to structure maps (\Cref{prop:inclusionmaps}); because the functor must preserve structure maps, the image of these is determined.
  
  We have built then an assignment on objects and morphisms. The assignment must be functorial because there exists at most a unique morphism between any two distinctly typed objects in the free \physicalDuoidalCategory{}. This creates the only possible functor, $f^{\ast} ፡ \Zetless(A) → 𝕍$, satisfying the factorization property.
\end{proof}

\begin{proposition}[From {{\Cref{prop:bracketedSubstituted}}}]
    \label{ax:prop:bracketedSubstituted}
    A \poset{} $R$ arises as a substitution $R = \subs{Q}{x}{P}$ of any of its full subposets, $P ⊆ R$, if and only if it is bracketed.
\end{proposition}
\begin{proof}
    The subposet $P ⊆ \subs{Q}{x}{P}$ is bracketed by construction. Let us show that any \poset{} with a bracketed subposet is the result of a substitution. Indeed, if $P ⊆ R$ is bracketed, then we can construct a poset $Q$ that contains all of the objects of $R$ but a single object substituting these on $P$,
    \[
      Q = (R_{obj} - P_{obj} + \{x\}; ≤_R + \{ r ≤ x \mid r ≤ p, p ∈ P\}) + \{ x ≤ r \mid p ≤ r, p ∈ P\}).
    \]
    We can see now that $R = \subs{Q}{x}{P}$. Indeed, if $r ≤ p₀$ in the original \poset{} then $r ≤ p$ for each $p ∈ P$; but then $r ≤ x$ and we indeed recover $r ≤ p$. Analogously, if $p₀ ≤ r$ in the original \poset{} then $p ≤ r$ for each $p ∈ P$; but then $x ≤ r$ and we indeed recover $p ≤ r$.
  \end{proof}

  \begin{proposition}[From {{\Cref{prop:subsetBracketed}}}]
    \label{ax:prop:subsetBracketed}
    A subset of a \zetlessPoset{} is an \interval{} if and only if it appears as a bracketed poset in some saturation of the \poset{}.
  \end{proposition}
  \begin{proof}
    Any bracketed poset must be an interval; and any interval remains an interval in a less saturated poset. Let us prove that any interval $I ⊆ P$ appears as a bracketed poset after some saturation. Indeed, for each $u ∈ P$ such that there exists $x₀ ∈ I$ with $u ≤ x₀$, we impose that $u ≤ x$ for every $x ∈ I$; for each $u ∈ P$ such that there exists $x₀ ∈ I$ with $x₀ ≤ u$, we impose that $x ≤ u$ for every $x ∈ I$. Because $I$ is an interval, this constitutes a saturation that does not identify any two elements of the poset; this saturation makes the interval bracketed.
  \end{proof}
  
  \begin{proposition}[From {{\Cref{prop:stringDiagramsFunctor}}}]
    \label{ax:prop:stringDiagramsFunctor}
    The construction of physical duoidal string diagrams over a \physicalDuoidalSignature{} extends to a functor $$\phyString ፡ \PhySig → \PhyDuo.$$  
  \end{proposition}
  \begin{proof}
    The functor will take a \signatureHomomorphism{} $f ፡ 𝓖 → 𝓗$ into a \strictPhysicalDuoidalFunctor{} $\phyString(f) ፡ \phyString(𝓖) → \phyString(𝓗)$. Let us first define what $\phyString(f)$ does: it takes any \stringDiagram{} into the same \stringDiagram{} but relabelling the nodes and wires using $f$ and $fₜ$, respectively: this is a compatible relabelling because a \signatureHomomorphism{} must preserve the source and target of its generators, $𝓖(U;V)$ is mapped into $𝓗(fU;fV)$.
  
    We need to show that $\phyString(f)$, as defined here, forms a \strictPhysicalDuoidalFunctor{}. Because the relabelling of a composition, sequencing or tensoring of diagrams is equal to the composition, sequencing or tensoring of its relabeled components, we know that it must form a \strictPhysicalDuoidalFunctor{}.
  
    Finally, we can check that relabelling by a composition of functions is the same as relabelling twice by each of the functions; relabelling by an identity yields the identity on string diagrams. This makes our original construction functorial.
  \end{proof}

  \begin{proposition}[From {{\Cref{prop:parprime-connected}}}]
    \label{ax:prop:parprime-connected}
    A \poset{} is \parPrime{} if and only if it is connected.
  \end{proposition}
  \begin{proof}
    If a \poset{} is disconnected, then it can be written as the tensor if its connected components, contradicting primality. If a poset is not $⊗$-prime, it must be disconnected because vertices on its different factors cannot be connected.
  \end{proof}
  
  \begin{proposition}[From {{\Cref{prop:seqprime-incomparable}}}]
    \label{ax:prop:seqprime-incomparable}
    A \poset{} $P$ is \seqPrime{} if and only if it is \incomparableConnected{}.
  \end{proposition}
  \begin{proof}
    Assume the \poset{} $P$ has more than one \incomparableConnected{} component: $P₁,...,Pₙ$; with all of them being non-empty. Picking any two components and any two elements on them, assume without loss of generality that $pᵢ ∈ Pᵢ$ and $pⱼ ∈ Pⱼ$ give $pᵢ ≤ pⱼ$.
  
    Now, if $pᵢ' ≤ pᵢ$, it cannot be that $pⱼ ≤ pᵢ'$ because that would imply $pⱼ ≤ pᵢ$; it must be that $pᵢ' ≤ pⱼ$: thus, all \incomparableConnected{} elements must be below $pⱼ$, and we conclude $Pᵢ < pⱼ$. By an analogous reansoning, $Pᵢ < Pⱼ$. This imposes a total order on the components. If, without loss of generality, we assume that $P₁ < ... < Pₙ$, then we can conclude
    $P = P₁ ⊲ ... ⊲ Pₙ,$ contradicting primality.
  
    Finally, if a \poset{} is written as a sequencing of posets, it must be incomparable-disconnected, as each factor is a different component.
  \end{proof}
    
  \begin{proposition}
    \label{ax:prop:prime-prime-singleton}
    Any \seqPrime{} and \parPrime{} \poset{} is a singleton.
  \end{proposition}
  \begin{proof}
    Primality means that the \poset{} must be connected (\Cref{prop:parprime-connected}) and incomparable-connected (\Cref{prop:seqprime-incomparable}). Let us assume the \poset{} is not a singleton and arrive to a contradiction. If it is not a singleton, it must contain a minimal and maximal element that are distinct but, because of connectedness, related as $i < m$.
  
    Now, they must be \incomparableConnected{}, and we are going to prove that any element $(∥)$-connect to $o$ cannot be $(∥)$-connected to $m$: it must be below $m$. There are two cases here.
    \begin{itemize}
      \item Connected by a single step: let $o ∥ u$. They must be connected by a span, $o → v ← u$; but then we require $u ≤ m$ to break the zet.
      \item Connected by multiple steps: let it be $i ∥ u₁ ∥ u₂ ∥ ... ∥ uₙ$. We will prove by induction on the length $n$ that $uₙ$ is below $m$. We know that it cannot be that $u_{n-2} ∥ u$, because that would contradict minimality: that means that either $uₙ < u_{n-2}$ (and, in that case, $u < m$) or $u_{n-2} < uₙ$ (and, to break the zet, $uₙ < m$). In any case, $uₙ < m$ is forced.
    \end{itemize}
    Because anything \incomparableConnected{} to $n$ is below $m$, it will never $(∥)$-connect to $m$; this reaches a contradiction.
  \end{proof}
  
  \begin{proposition}[From {{\Cref{prop:forgetFunctor}}}]
    \label{ax:prop:forgetPhysical}
    \label{ax:prop:forgetFunctor}
    Forgetting about the composition, sequencing and tensoring of a \physicalDuoidalCategory{} extends to a functor $$\Forget ፡ \PhyDuo → \PhySig.$$
  \end{proposition}
  \begin{proof}
    Let $𝕍$ be a \physicalDuoidalCategory{}. The basic types of the associated \physicalDuoidalSignature{} are all of the objects of the category, $\Forget(𝓖)ₜ = 𝕍_{obj}$. For each pair of \duoidalExpressions{}, $Eᵢ, E_o ∈ \Expr(𝓖ₜ)$, we pick the set of morphisms $𝕍(⟦Eᵢ⟧,⟦E_o⟧)$ as our generators.
  \end{proof}

  \begin{proposition}[From {{\Cref{prop:sequencingTensoringZetless}}}]
    \label{ax:prop:sequencingTensoringZetless}
    The sequencing and tensoring of two \zetlessPosets{} is again a \zetlessPoset{}.
  \end{proposition}
  \begin{proof}
    Because $\mathsf{Z}$ is connected, if it embeds into the disconnected poset $P ⊗ Q$, the embedding must be contained into any of the connected components.  
    Because $\mathsf{Z}$ is \incomparableConnected{}, if it embeds into the incomparable-disconnected poset $P ⊲ Q$, the embedding must be contained into any of the \incomparableConnected{} components.
  \end{proof}

  \begin{theorem}[From {{\Cref{th:diagramsPhysical}}}]
    \label{ax:th:diagramsPhysical}
    \defining{linkPhyString}{}
    Physical duoidal \stringDiagrams{} over a \physicalDuoidalSignature{}, $𝓖$ form a \physicalDuoidalCategory{}, $\mathsf{PhyString}(𝓖)$.
  
    Objects are \duoidalExpressions{} over the types of the \physicalDuoidalSignature{}; morphisms are hypergraphs with input and output given by the \zetlessPosets{} corresponding to the \duoidalExpressions{}.
    \begin{figure}[ht]
      \centering
      \includegraphics[width=.8\textwidth]{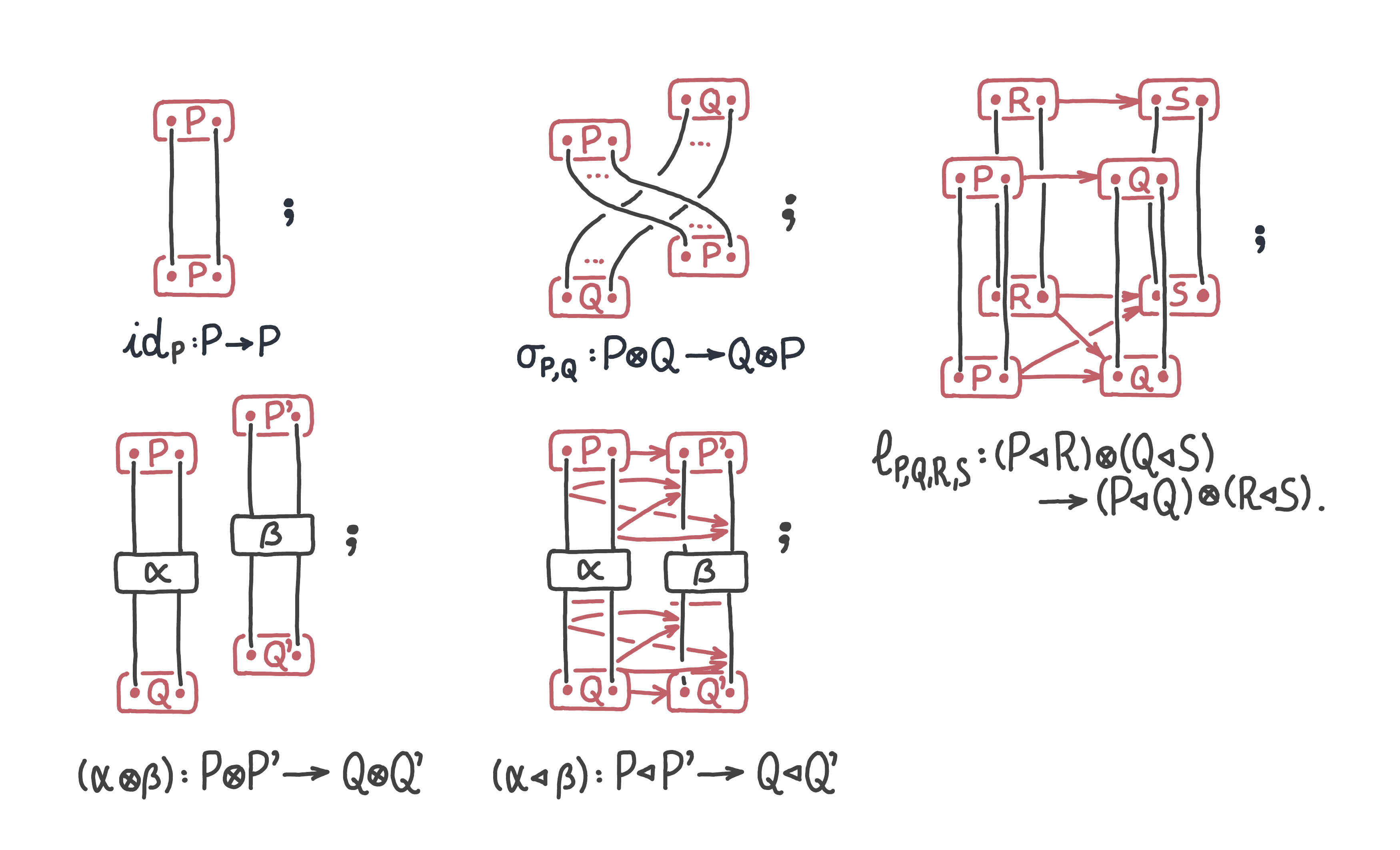}
      \caption{Physical string diagrams form a physical duoidal category.}
      \label{ax:fig:phystring:phycategory}
    \end{figure}
  
    Composing two string diagrams, $α$ and $β$, concatenates them, so that the output wires of $α$ and the input wires of $β$ get merged into single wires. Parallel tensoring juxtapos two diagrams; sequential tensoring juxtaposes but also links every wire from the first diagram to the second diagram (see \Cref{fig:phystring:phycategory}).
  \end{theorem}
  \begin{proof}
    The category structure is given by composition of string diagrams (\Cref{def:compositionStringDiagrams}), tensoring (\Cref{def:tensoringStringDiagrams}), and sequencing (\Cref{def:sequencingStringDiagrams}). Composition is associative (\Cref{prop:compositionStringDiagramsIsAssociative}), and we can check it is unital; tensoring and sequencing are unital and associative. The interchange equation for duoidal categories holds. We can construct the laxators by a \stringDiagram{} consisting only of wires.
  \end{proof}

  \begin{definition}[Composition of string diagrams]
    \label{def:compositionStringDiagrams}
    Let $α ፡ E₁ → E₂$ and $β ፡ E₂ → E_3$ be two \stringDiagrams{}. Their composition, $(α ⨾ β) ፡ E_1 → E_3$, has all the nodes of both string diagrams,
    $$\nodes(α ⨾ β) = \nodes(α) + \nodes(β);$$
    and the pushout of wires of string diagrams along the output and input boundaries; in other words, it is the pushout of $\Input(o_{α}) ፡ E_2 → \Wires(α)$ and $\Output(i_{β}) ፡ E₂ → \Wires(β)$, written as
    $$\wires(α ⨾ β) = \wires(α) +_{\Input(o_{α})}^{\Output(i_{β})} \wires(β).$$
    The source and target functions of the nodes are preserved. The composition of wire-linear and acyclic hypergraphs is again a wire-linear and acyclic hypergraph: by definition, there does not exist any wire with source in $β$ and target in $α$; the only wires that change source and target were output and input wires that got merged, going from a single node in $α$ to a single node in $β$.
  \end{definition}
  
  \begin{proposition}[Composition of string diagrams is associative]
    \label{prop:compositionStringDiagramsIsAssociative}
    Let $α ፡ E₁ → E₂$, $β ፡ E₂ → E_3$, and $\gamma ፡ E_3 → E_4$ be three \stringDiagrams{}. Their composition is associative, $(α ⨾ β) ⨾ \gamma = α ⨾ (β ⨾ \gamma)$.
  \end{proposition}
  \begin{proof}
    The only observation we need is that the construction of the pushout is associative,
    \begin{align*}
      & \big(\wires(α) +_{\Input(o_{α})}^{\Output(i_{β})} \wires(β) \big) +_{\Input(o_{β})}^{\Output(i_{\gamma})} \wires(\gamma) ≅ \\
      & \wires(α) +_{\Input(o_{α})}^{\Output(i_{β})} \big(\wires(β) +_{\Input(o_{β})}^{\Output(i_{\gamma})} \wires(\gamma)\big).
    \end{align*}
    The source and target functions of the nodes are preserved in any case.
  \end{proof}

  \begin{definition}[Tensoring of string diagrams]
    \label{def:tensoringStringDiagrams}
    Let $α ፡ E₁ → E_3$ and $β ፡ E₂ → E_4$ be two \stringDiagrams{}. Their tensoring, $(α ⊗ β) ፡ E_1 ⊗_e E_2 → E_3 ⊗_e E_4$, has all the nodes and wires of both string diagrams,
    \begin{align*}
      \nodes(α ⊗ β) & = \nodes(α) + \nodes(β); \\
      \wires(α ⊗ β) & = \wires(α) + \wires(β).
    \end{align*}
    The source and target functions of the nodes are preserved. The poset structure is preserved. The tensoring of wire-linear and acyclic hypergraphs is again a wire-linear and acyclic hypergraph: by definition, there does not exist any wire with source in $β$ and target in $α$; no wires change source or target.
  \end{definition}

  \begin{definition}[Sequencing of string diagrams]
    \label{def:sequencingStringDiagrams}
    Let $α ፡ E₁ → E_3$ and $β ፡ E₂ → E_4$ be two \stringDiagrams{}. Their sequencing, $(α ⊲ β) ፡ E_1 ⊲_e E_2 → E_3 ⊲_e E_4$, has all the nodes and wires of both string diagrams,
    \begin{align*}
      \nodes(α ⊲ β) & = \nodes(α) + \nodes(β); \\
      \wires(α ⊲ β) & = \wires(α) + \wires(β).
    \end{align*}
    The source and target functions of the nodes are preserved. However, the poset structure changes: we construct a new poset with all the previous inequalities but additionally imposing that $x ⊑ y$ whenever $x ∈ α$ and $y ∈ β$. This is still a poset because elements from both diagrams were previously incomparable.
    The sequencing of wire-linear and acyclic hypergraphs is again a wire-linear and acyclic hypergraph: by definition, there does not exist any wire with source in $β$ and target in $α$; no wires change source or target.
  \end{definition}

  \begin{proposition}[From {{\Cref{prop:physicalHypergraphAtomic}}}]
    \label{ax:prop:physicalHypergraphAtomic}
    Every \physicalHypergraph{} can be written as the composition of atomic hypergraphs.
  \end{proposition}
  \begin{proof}
    Every \physicalHypergraph{} is acyclic, and nodes can be given a total ordering, $(≼)$, even if not uniquely (\Cref{def:orderingNodes}). If nodes are ordered as 
    \[n₀ ≼ n₁ ≼ n₂ ≼ ... ≼ nₖ,\]
    then we claim that the following composition recovers the hypergraph,
    \[
    \Atomic(n₀) ⨾ \Atomic(n₁) ⨾ \Atomic(n₂) ⨾ ... ⨾ \Atomic(nₖ).
    \]
    We can proceed by induction on the number of nodes. If a \physicalHypergraph{} has zero nodes, then it consists of only wires and it is the identity for composition. 
    
    If a \physicalHypergraph{} has $k + 1$ nodes, we can pick the first of them, $n₀$, and realize that all of its input wires must be part of the input to the hypergraph: if they were connected to the output of any node, that would contradict the properties of the ordering of nodes, $(≼)$. We can thus form a new hypergraph by removing the first node and all its inputs, and by taking the outputs of the node to be inputs of the hypergraph: this is still acyclic and we have not removed any wire connected to another node. By induction hypothesis, this new hypergraph can be written as follows,
    $$\Atomic(n_1) ⨾ \Atomic(n_2) ⨾ ... ⨾ \Atomic(n_k).$$
    It suffices to check that we constructed it precisely so that composing with $\Atomic(n₀)$ recovers the initial graph: it adds the node $n₀$, reconnects the output wires to the outputs of $n_0$, and substitutes them by the inputs of $n₁$ in the inputs to the hypergraph.
  \end{proof}

  \begin{lemma}[From {{\Cref{lemma:universalPhyString}}}]
    \label{ax:lemma:universalPhyString}
    For each \physicalDuoidalSignature{}, $𝓖$, there exists a \signatureHomomorphism{}, $u_{𝓖} ፡ 𝓖 → \Forget(\phyString(𝓖))$.
  \end{lemma}
  \begin{proof}
    Each generator, $f ∈ 𝓖(E_i; E_o)$, gets sent to the \stringDiagram{} consisting only of a single node labelled by the generator, going from the \zetlessPoset{} given by $\Encode(E_i)$ to the \zetlessPoset{} given by $\Encode(E_o)$.
  \end{proof}
  
  \begin{theorem}[From {{\Cref{thm:freeness}}}]
    \label{ax:thm:freeness}
    There exists an adjunction from the category of \physicalDuoidalSignatures{} to the category of \strictPhysicalDuoidalCategories{} given by physical duoidal string diagrams $\phyString ፡ \PhySig → \PhyDuo$ and the forgetful functor, $\mathsf{forget} ፡ \PhyDuo → \PhySig$.
  \end{theorem}
  \begin{proof}
    We will show that the morphisms $u_{𝓖} ፡ 𝓖 → \Forget(\phyString(𝓖))$ defined in \Cref{lemma:universalPhyString} are universal.
    Let $f ፡ 𝓖 → \Forget(𝕍)$ be any \signatureHomomorphism{}. We will show that there exists a unique \strictPhysicalDuoidalFunctor{} $F ፡ \phyString(𝓖) → 𝕍$ such that $f = u_{𝓖} ⨾ \Forget(F)$.
  
    Let us first define the functor on objects. The objects of $\phyString(𝓖)$ are \duoidalExpressions{} labelled by the basic types of the signature (by virtue of \Cref{prop:zetlessPosetsDuoidalExpressions}). As a consequence, $F(E)$ is determined uniquely for each $E ∈ \Expr(𝓖ₜ)$.
  
    Let us now define the functor on morphisms. The morphisms of $\phyString(𝓖)$ are \physicalHypergraphs{}, which can be decomposed into atomic hypergraphs (\Cref{prop:physicalHypergraphAtomic}). Any atomic hypergraph, $\Level(n)$, consists of a generator, $\Label(n)$, tensored and sequenced with parallel wires: it must be mapped to $F(\Label(n))$ tensored and sequenced with parallel wires. Finally, because the functor preserves composition, the value on any \physicalHypergraph{} is determined by its decomposition into atomic hypergraphs.
  
    It suffices to check that this assignment is well-defined and that it forms a \strictPhysicalDuoidalFunctor{}. Firstly, let us note that, if we were to pick a different ordering of the nodes, then the only nodes that would change ordering are these that are parallel: they cannot share any wire. In that case, by the interchange law for duoidal categories, we know that the value of both interpretations in the target category is the same.
  
    Finally, we need to prove that this forms a \strictPhysicalDuoidalFunctor{}. Note that it must preserve compositions, for the concatenation of decompositions into atomic graphs is a possible decomposition of the composition. It must preserve tensoring and sequencing, for there is always an ordering that gets all of the nodes of one of the \hypergraphs{} first, and there, again, the concatenation of decompositions into atomic \hypergraphs{} (this time with the extra input or output wires of the other hypergraph) is a possible decomposition of the tensoring or sequencing.
  \end{proof}

\end{document}